\documentclass[reqno]{amsart}
\usepackage{hyperref}
\usepackage{amssymb,amsmath,amsthm}
\usepackage{amssymb,amsmath}
\usepackage{amsfonts}
\usepackage{graphicx,epsfig}
\usepackage{url}
\newtheorem{theorem}{Theorem}
\newtheorem{lemma}{Lemma} 
\newtheorem{fact}[lemma]{Fact}

\newtheorem{claim}[lemma]{Claim}
\theoremstyle{definition}
\newtheorem{example}[lemma]{Example}
\def\Z{{\mathbb{Z}}} 
\def\N{{\mathbb{N}}} 
\def\cA{{\mathcal A}}
\def\cB{{\mathcal B}}
\def\cF{{\mathcal F}}
\def\cG{{\mathcal G}}
\def\cW{{\mathcal W}}

\def\dual{{{\rm dual}_t}}
\def\shiftsto{\to}

\begin{document}
\title
{AK-type stability theorems on cross $t$-intersecting families}
\thanks{{This is a  preprint version of the paper appearing in European J. Combinat. 82 (2019) 102993.  \url{https://doi.org/10.1016/j.ejc.2019.07.004}}}
\author{Sang June Lee}
\address{Department of Mathematics, Kyung Hee University, Seoul 02447, South Korea}
\email{sjlee242@khu.ac.kr}
\thanks{The first author was supported by Korea Electric Power Corporation (Grant number:R18XA01) and by Basic Science Research Program through the National Research Foundation of Korea (NRF) funded by the Ministry of Education (NRF-2016R1D1A1B03933404 and NRF-2019R1F1A1058860). The second author is supported by Korean NRF Basic Science Research Program (NRF-2018R1D1A1A09083741) and by the  Kyungpook National University Research Fund. The last author was supported by JSPS KAKENHI 25287031, 18K03399}
\author{Mark Siggers}
\address{College of Natural Sciences, Kyungpook National University, Daegu 702-701, South Korea}
\email{mhsiggers@knu.ac.kr}
\author{Norihide Tokushige}
\address{College of Education, Ryukyu University, Nishihara, Okinawa 903-0213, Japan}
\email{hide@edu.u-ryukyu.ac.jp}
\keywords{Cross intersecting families; Erd\H{o}s-Ko-Rado theorem; Ahlswede-Khachatrian theorem; Shifting; Random walks}

\begin{abstract}
  Two families, ${\mathcal A}$ and ${\mathcal B}$, of subsets of $[n]$ are
  cross $t$-intersecting if for every $A \in {\mathcal A}$ and $B \in {\mathcal B}$,
  $A$ and $B$ intersect in at least $t$ elements. 
  For a real number $p$ and a family ${\mathcal A}$ the product measure $\mu_p ({\mathcal A})$
  is defined as the sum of $p^{|A|}(1-p)^{n-|A|}$  over all
  $A\in{\mathcal A}$.
  For every non-negative integer $r$, and for large enough $t$, we determine, for any $p$ satisfying
  $\frac r{t+2r-1}\leq p\leq\frac{r+1}{t+2r+1}$, the maximum possible value of
  $\mu_p ({\mathcal A})\mu_p ({\mathcal B})$ for cross $t$-intersecting families ${\mathcal A}$ and ${\mathcal B}$.
  In this paper we prove a stronger stability result which yields the
above result.
\end{abstract}

\maketitle
\section{Introduction}
Let $n\geq t$ be positive integers. Let $[n]:=\{1,2,\ldots,n\}$ and 
$2^{[n]}:=\{F:F\subset[n]\}$.
A family of subsets $\cA\subset 2^{[n]}$ is called \emph{$t$-intersecting} 
if $|A\cap A'|\geq t$ for all $A,A'\in\cA$.
For any real number $p\in(0,1)$ and a family $\cA\subset 2^{[n]}$, 
we define the \emph{product measure} $\mu_p $ of $\cA$ by 
\[
 \mu_p (\cA):=\sum_{A\in\cA}p^{|A|}(1-p)^{n-|A|}.
\]
What is the maximum product measure of $t$-intersecting families? 
To answer this question, let $r$ be a non-negative integer and let
\[
 \cF_r^t:=\{F\subset[n]:|F\cap[t+2r]|\geq t+r\}. 
\]
The family $\cF_r^t$ is $t$-intersecting since $|F\cap F'\cap[t+2r]|\geq t$
for all $F,F'\in\cF^t_r$. Two families $\cA,\cA'\subset 2^{[n]}$
are isomorphic, denoted by $\cA\cong\cA'$, if 
$\cA'=\{\{\pi(a):a\in A\}:A\in\cA\}$, where $\pi$ is a permutation on $[n]$.
Answering a conjecture of Frankl, and extending partial results by Frankl and F\"uredi in \cite{FF},
the following result is essentially proved in \cite{AK-p} by Ahlswede and
Khachatrian,  see also \cite{BE,DS,Filmus,Tuvsw}.

\begin{theorem}
If $\cA\subset 2^{[n]}$ is $t$-intersecting, then
\[
 \mu_p (\cA)\leq\max_{0\leq r\leq \frac{n-t}2}\mu_p (\cF^t_r).
\]
Moreover, equality holds if and only if $\cA\cong\cF^t_r$ for some $r$.
\end{theorem}
Grouping the subsets in the family $\cF^t_r$ according to the size of
their intersection with $[t+2r]$ we see, where $q = 1-p$, 
\begin{equation}\label{weight cFtr1}
  \mu_p (\cF_r^t) = \sum_{i = 0}^r \binom{t+2r}{i}p^{t + 2r - i}q^i.
\end{equation}  
By comparing $\mu_p (\cF^t_{r+1}\setminus\cF^t_r) = \binom{t+2r}{r+1} p^{t+r+1}q^{r+1}$ and 
$\mu_p (\cF^t_r\setminus\cF^t_{r+1})= \binom{t+2r}r p^{t+r}q^{r+2}$, one sees that 
$\mu_p (\cF^t_{r+1})-\mu_p (\cF^t_{r})$ is positive, $0$, negative 
if and only if $p-\frac{r+1}{t+2r+1}$ is positive, $0$, negative, respectively.
In particular, if 
\begin{equation}\label{def:p range}
 \frac r{t+2r-1}\leq p\leq \frac{r+1}{t+2r+1}, 
\end{equation}
then 
\[
 \mu_p (\cF^t_0)< \mu_p (\cF^t_1)<\cdots< \mu_p (\cF^t_{r-1})\leq
\mu_p (\cF^t_{r})\geq\mu_p (\cF^t_{r+1})>\mu_p (\cF^t_{r+2})>\cdots.
\]
Thus the Ahlswede--Khachatrian theorem says that the maximum product measure
of $t$-intersecting families is given by $\mu_p (\cF^t_r)$ provided that $t,p,$ and $r$ satisfy \eqref{def:p range}.

We extend this result to two families in $2^{[n]}$. Two families 
$\cA,\cB\subset 2^{[n]}$ are \emph{cross $t$-intersecting} if $|A\cap B|\geq t$ 
for all $A\in\cA$ and $B\in\cB$.
In this case, it is conjectured in \cite{LST} that 
\begin{equation}\label{eq:conjecture}\mu_p (\cA)\mu_p (\cB)\leq\mu_p (\cF^t_r)^2\end{equation} under the assumption of 
\eqref{def:p range}. The inequality~\eqref{eq:conjecture} was proved for $r=0$ and $n\geq t\geq 14$ in~\cite{FLST} and for $r=1$ and $n\geq t\geq 200$ in~\cite{LST}. We also mention that Borg obtained related results in 
\cite{Borg14,Borg16,Borg17}.

In this paper, using the random-walk method that was introduced by Frankl in \cite{Fr78,Fr87}, we verify
that this conjecture holds for every fixed $r$ and large $t$.
That is we prove the following, referring to \cite{FLST} for the case $r=0$.

\begin{theorem}\label{thm:inequality}
  For every integer $r\geq 0$, there exists an integer $ t_0 = t_0 (r)$, depending only on $r$,
  such that for all $n\geq t\geq  t_0 (r)$ and all $p$ with
  $\frac {r}{t+2r-1}\leq p\leq \frac {r+1}{t+2r+1}$, the following
  holds. 
If $\cA, \cB \subset 2^{[n]}$ are cross $t$-intersecting, then
\begin{align}\label{eq:main ineq}
 \mu_p (\cA)\mu_p (\cB) &\leq \mu_p (\cF^t_r)^2.
\end{align}
Moreover, equality holds if and only if one of the following holds:

\begin{enumerate}
\item $\cA=\cB\cong \cF^t_{r-1}$ and $p=\dfrac{r}{t+2r-1}$,
\item $\cA=\cB\cong \cF^t_r$ and $\dfrac{r}{t+2r-1}\leq p\leq \dfrac{r+1}{t+2r+1}$, or
\item $\cA=\cB\cong \cF^t_{r+1}$ and $p=\dfrac{r+1}{t+2r+1}$.
\end{enumerate}
\end{theorem}

In this paper we do not attempt to optimize $ t_0 (r)$. 
We simplify calculations by assuming,  for each fixed $r$, that $t$ is sufficiently large.   
As such, when  we use asymptotic notation such
as $o(1)$, $O(f)$, or $f \ll g$, it is always asymptotic in $t$, with $r$ fixed,
and $p$ being some fixed proportion of the way through the range \eqref{def:p range}. 

We also consider the stability of extremal structures. Suppose that
$\cA,\cB\subset 2^{[n]}$ are cross $t$-intersecting families.
If condition \eqref{def:p range} is satisfied and $\mu_p (\cA)\mu_p (\cB)$ 
is close to the maximum value, then we can ask whether $\cA$ and $\cB$ are
close to (isomorphic copies of) the extremal families $\cF^t_{r-1}, \cF^t_{r}$,
or $\cF^t_{r+1}$, where we say that  two families $\cF$ and $\cG$ are close
if their symmetric difference 
$\cF\triangle\cG=(\cF\setminus\cG)\cup(\cG\setminus\cF)$ has small measure.
We are able to show that this is true if the
two families satisfy the additional condition of being `shifted.'

A family $\cA\subset 2^{[n]}$ is \emph{shifted} (sometimes called {\em compressed})
if   $A\in\cA$ and
$\{i,j\}\cap A=\{j\}$ for some $1\leq i<j\leq n$ imply
$(A\setminus\{j\})\cup\{i\}\in \cA$.
It is known (see, e.g., Lemma 2.3 in \cite{FLST}) that for any given
cross $t$-intersecting families $\cA,\cB\subset 2^{[n]}$ one can apply a
sequence of shifting operations to them and get shifted 
cross $t$-intersecting families $\cA',\cB'\subset 2^{[n]}$ such that
$\mu_p (\cA)=\mu_p (\cA')$ and $\mu_p (\cB)=\mu_p (\cB')$.
Notice that the definition of a shifted family depends on the ordering
of $[n]$, so an isomorphic copy of a shifted family is not
necessarily shifted in this sense.

A family $\cA$ is \emph{inclusion maximal} if $A\subset A'$ and $A\in\cA$
imply $A'\in\cA$ as well. Since we are interested in the maximum measure of cross $t$-intersecting families, we always assume that families are inclusion maximal.
It is not difficult to see that the property of
being inclusion maximal is invariant under shifting operations.

Two families $\cA$ and $\cB$ are \emph{$t$-nice} if they are shifted, 
inclusion maximal, and cross $t$-intersecting.
We obtain the following statement of stability.

\begin{theorem}\label{thm:weakstability}
For every integer $r\geq 0$ and all real numbers $\epsilon \in (0,1/2)$ and $C>2$, there exists an integer $ t_0 = t_0 (r,\epsilon, C)$ such that
for all $n\geq t\geq  t_0 $ 
the following holds.
Let 
$\frac {r+\epsilon}{t}\leq  p\leq \frac {r+1-\epsilon}{t}$, and let $\cA,\cB\subset 2^{[n]}$ be $t$-nice families.
If 
\begin{equation}\label{eq:AB_large(2)}\mu_p (\cA)\mu_p (\cB)\geq (1-\delta)\mu_p (\cF^t_r)^2\end{equation} with 
$\delta\in (0,\epsilon/(r+1))$, then
\begin{equation}\label{eq:<f(gamma)}
\mu_p (\cA\triangle\cF^t_r)+\mu_p (\cB\triangle\cF^t_r)\leq C\left(1-\sqrt{1-\delta}\right)\mu_p(\cF^t_r).
\end{equation}
\end{theorem}

For $r=0$, a similar result was proved in~\cite{FLST}; for the same $t$ it is weaker than this, but it is proved for $t_0 = 14$. 
There are some points about this theorem that bear further explanation.
For one, no matter how close we require that $\mu_p (\cA)\mu_p (\cB)$
be to $\mu_p (\cF_r^t)^2$ in \eqref{eq:AB_large(2)} there are $t$-nice
families $\cA$ and $\cB$, which are \emph{not}
subfamilies of $\cF_r^t$, that satisfy \eqref{eq:<f(gamma)}.
Indeed, consider the families 
\[
 \cA=\cB=\left(\cF^t_r\setminus \tbinom{[t+2r]}{t+r} \right)\sqcup
\Big\{G\sqcup[t+2r+1,n]:G\in\tbinom{[t+2r]}{t+r-1}\Big\}.
\]
The families $\cA$ and $\cB$ are $t$-nice, and 
$\mu_p(\cA)\mu_p(\cB)\to\mu_p(\cF^t_r)^2$ for $n \gg t$ as $t\to\infty$.
Theorem~\ref{thm:weakstability} says that such families  $\cA$ and $\cB$ must be close to $\cF^t_r$ in the sense that the sum of their measures 
$\mu_p(\cA\triangle\cF^t_r)+\mu_p(\cB\triangle\cF^t_r)$ goes to $0$.
Observe that with such a definition of closeness, inequality~\eqref{eq:<f(gamma)} is sharp.
Indeed, if $\cA=\cB\subset\cF^t_r$ and 
$\mu_p(\cA)=\sqrt{1-\delta}\mu_p(\cF^t_r)$ then the LHS of \eqref{eq:<f(gamma)} 
is precisely $2(1 - \sqrt{1 - \delta})\,\mu_p(\cF^t_r)$.

Another point we should explain is the reduced range of $p$ in the statement of the theorem.  
When $p = (r+1)/(t + 2r +1)$ we have $\mu_p (\cF^t_{r+1}) = \mu_p (\cF^t_r)$, so to make a statement of
stability with respect to $\cF^t_r$ it is necessary to move $p$ away from this point. We thus introduce
a gap of $\epsilon$ into the bound $p < (r+1-\epsilon)/(t+2r+1)$, and once this is introduced, we
absorb constants into it and simplify it to $p < (r+1-\epsilon)/t$.
Similarly, we require $p>(r+\epsilon)/t$ because of the family $\cF^t_{r-1}$.
As $\epsilon$  goes to $0$, we must introduce $\delta$ in \eqref{eq:AB_large(2)} that depends on $\epsilon$.  
 
It turns out that the condition $\delta<\epsilon/(r+1)$ is sharp.
Indeed, consider the following pair of families:
\begin{center}
 $\cA=\cF^{t-1}_{r+1}$ and $\cB=\cF^{t+1}_r$. 
\end{center}
The families $\cA$ and $\cB$ are $t$-nice, and  far from $\cF^t_r$, 
in fact, both 
$\mu_p(\cA\setminus\cF^t_r)/\mu_p(\cF^t_r)$ and
$\mu_p(\cF^t_r\setminus\cB)/\mu_p(\cF^t_r)$ go to infinity as $t\to\infty$.
On the other hand, one 
can show that if $p=\frac{r+1-\epsilon}{t}$, then
\[\frac{\mu_p(\cA)\mu_p(\cB)}{\mu_p(\cF_r^t)^2}= \left(1 - \frac{\epsilon}{r+1}\right)(1+o(1)) 
\] (See~\eqref{eq:<1-e/(r+2)}).
So the condition $\delta < \epsilon/(r+1)$ cannot be improved.

For a full stability result, we must consider these other extremal
families. We do this with the more complicated
Theorem~\ref{thm:stability},
from which Theorem~\ref{thm:weakstability} follows as a corollary.  
Assume that $\cA$ and $\cB$ are $t$-nice families, 
and \eqref{def:p range} is satisfied.
Theorem~\ref{thm:stability} says that 
either $\mu_p (\cA)\mu_p (\cB)$ is much smaller than the optimal value $\mu_p (\cF_r^t)^2$, or $\{\cA,\cB\}$ is close to one of
\begin{equation}\label{eq:extremal families}
\text{
$\{\cF^t_{r-1},\cF^t_{r-1}\}$, $\{\cF^t_r,\cF^t_r\}$, 
$\{\cF^t_{r+1},\cF^t_{r+1}\}$, $\{\cF^{t-1}_r,\cF^{t+1}_{r-1}\}$, 
or $\{\cF^{t-1}_{r+1},\cF^{t+1}_{r}\}$.
}
\end{equation}

We  denote the set of pairs of subscripts of the extremal families in \eqref{eq:extremal families} 
 by
  \begin{equation}\label{def:R_ex} R_{\rm ex}:= \left\{
      \begin{array}{cl}
          \{(0,0),(1,1),(1,0)\} & \text{if } r = 0\\
          \{(r-1,r-1),(r,r),(r+1,r+1),(r,r-1),(r+1,r)\} & \text{if } r\geq 1.
      \end{array} \right.
 \end{equation}

 If $\mu_p (\cA)\mu_p (\cB)$ is close to optimal, then we will have, up to switching $\cA$ and $\cB$,
 that  $(\cA,\cB)$ is close to $(\cF^u_s,\cF^v_{s'})$ for a unique $(s,s') \in R_{\rm ex}$, with
 $u = t - (s-s')$ and $v = t + (s - s')$.  Again, to quantify `close to' we consider the measure
 of the symmetric differences
 $\cA\triangle\cF^u_s$ and $\cB\triangle\cF^v_{s'}$.  We could just sum these, but we will
 observe below, in \eqref{weight cFtr2},  that $\mu_p (\cF^t_r) = \Theta(p^t)$, so the measures of
 $\cF_s^u$ and $\cF_{s'}^v$ can be vastly different. It is natural,
 therefore, to normalise the measures of these symmetric differences with respect to the measures of
$\cF^u_s$ and $\cF^v_{s'}$. 
We thus define the following normalised measures:
\begin{align*}
 X&:=p^{t-u}\mu_p (\cA) + p^{t-v}\mu_p (\cB), \\
X_{\cF}&:=p^{t-u}\mu_p (\cF^u_s)+p^{t-v}\mu_p (\cF^v_{s'}),\\
X_{\Delta}&:=p^{t-u}\mu_p (\cA\triangle\cF^u_s) + p^{t-v}\mu_p (\cB\triangle\cF^v_{s'}), \\
X_{*}&:=p^{t-u}\mu_p (\cA\setminus\cF^u_s) + p^{t-v}\mu_p (\cB\setminus\cF^v_{s'}).
\end{align*}
Now we can state our main result.
\begin{theorem}\label{thm:stability}
  For every integer $r\geq 0$, and all real numbers 
$\epsilon,\epsilon_1\in(0,1/2)$, and $\delta_1\in(0,1/(r+2))$,
there exists an integer 
\[
 t_0=\begin{cases}
      t_0(r,\epsilon_1,\delta_1)\text{ for }r\geq 1,\\
      t_0(\epsilon,\epsilon_1,\delta_1)\text{ for }r=0
     \end{cases} 
\]
such that for all $n\geq t\geq  t_0 $ the following holds. Let 
\begin{align*}
\frac {r}{t+2r-1}\leq p\leq \frac {r+1}{t+2r+1} \text{ for }r\geq 1,
\text{ and } \frac{\epsilon}{t}\leq p\leq\frac{1}{t+1} \text{ for }r=0,
\end{align*}
 and let $\cA,\cB\subset 2^{[n]}$ be 
$t$-nice families. 
Suppose \begin{equation}\label{eq:AB_large}\mu_p (\cA)\mu_p (\cB)\geq(1-\delta_1)\mu_p (\cF^t_r)^2.\end{equation} 
Then, up to switching $\cA$ and $\cB$, there exists unique $(s,s')\in R_{\rm ex}$ such that, for
\begin{equation*}\label{eq:(s,s') is in Rex}
u=t-(s-s') \mbox{ and }  v=t+(s-s'), 
\end{equation*}
the following hold.
\begin{enumerate}
\item[(a)]
$X\leq X_{\cF} $ with equality holding if and only if 
$\cA=\cF^u_s$ and $\cB=\cF^v_{s'}$, 
\item[(b)]
$X_{*}\leq \epsilon_1X_{\Delta}$, and
\item[(c)]
$\mu_p(\cA)\mu_p(\cB)\leq\mu_p(\cF^u_s)\mu_p(\cF^v_{s'})$ with equality holding if and only if 
$\cA=\cF^u_s$ and $\cB=\cF^v_{s'}$.
\end{enumerate}
\end{theorem}

What does this say exactly?  It says that if the product of the measures of $\cA$
and $\cB$ is close to optimal, then $(\cA, \cB)$ has measure not greater than one of the 
one of the pairs $(\cF_s^u, \cF_{s'}^v)$ of extremal families in \eqref{eq:extremal families}, and is 
close to this pair, in the sense that $X_\Delta$ is small. 
To see that $X_\Delta$ is small one can use (b) to show $X_{\Delta} \leq\frac{2(1-\sqrt{1-\delta_1})}{1-2\epsilon_1}\mu_p(\cF^t_r)$
as we do for  \eqref{eq:bndXdelta}.  
 That the measure of $(\cA, \cB)$ is at most that of $(\cF_s^u, \cF_{s'}^v)$ is stated with (a) and  (c). 
 We explain why there are two statements. The statements are very similar when
\eqref{eq:AB_large} holds. Statement (c) is perhaps the more obvious statement, in light of
the required inequality \eqref{eq:main ineq} of Theorem \ref{thm:inequality}, and in the case that
$s = s' = r$ it is all we need for proving both Theorem \ref{thm:inequality} and
\ref{thm:weakstability}.  Moreover, in this
case (c) follows from (a) by the AM-GM inequality.  When $s - s' = 1$ things are not so clean, and
we need both statements.  Statement (a) is stronger as $p$ approaches $r/(t+2r-1)$ or $(r+1)/(t+2r+1)$
and so we use this to prove Theorem \ref{thm:inequality}. When $p$ is bounded away from these
endpoints, (c) is stronger (and harder to prove) than statement (a). We use it to
prove Theorem \ref{thm:weakstability}.

We mention that the inequality $X\leq X_{\cF} $ in (a) is not necessarily true 
unless \eqref{eq:AB_large} holds. Indeed if, e.g., 
$\{\cA,\cB\}=\{\emptyset,2^{[n]}\}$,
or $\{\{[t]\},\bigcup_{i\geq t}\binom{[n]}t\}$, then it follows that $X\gg X_{\cF} $
if $n\gg t$.

We also mention that the condition $\delta_1<\frac1{r+2}$ is tight.
To see this, let $\{\cA,\cB\}=\{\cF^{t-1}_{r+2},\cF^{t+1}_{r+1}\}$,
or $\{\cF^{t-2}_{r+2},\cF^{t+2}_{r}\}$, and let $p=\frac{r+1}{t+2r+1}$.
 Then $\cA$ and $\cB$ are $t$-nice, and we have $\mu_p(\cA)\mu_p(\cB)/\mu_p(\cF^t_r)^2>1-\frac1{r+2}$,
  (and the LHS converges to the RHS as $t\to\infty$,)  but (b) does not hold for
  any $(s,s')\in R_{\rm ex}$.
This means that one cannot replace the condition \eqref{eq:AB_large}  with
$\mu_p(\cA)\mu_p(\cB)>(1-\frac1{r+2})\mu_p(\cF^t_r)^2$.

The organization of this paper is as follows.
In Section~\ref{sec:extFam} we 
derive Theorems~\ref{thm:inequality} and \ref{thm:weakstability} from
Theorem~\ref{thm:stability}.  Most of the rest of the paper is dedicated to the proof of
Theorem~\ref{thm:stability}. 
In Section \ref{sec:prelim} we recall some useful tools  from \cite{FLST} and \cite{LST}.
In Section \ref{sec:asymp} we lay out our asymptotic assumptions
and use them to give simplified expressions for frequently used values.
In Section \ref{subsec:setup}
we reduce Theorem~\ref{thm:stability} to the
essential case: $u+v=2t$. We then define the parameters $s$ and $s'$,
and use them to distinguish three cases for the proof of
Theorem~\ref{thm:stability}:  the non-extremal case, the diagonal extremal case, and the non-diagonal extremal case.
In Section~\ref{sec:ABEC} we settle the non-extremal case. 
We deal with the diagonal extremal case in
Section~\ref{sec:extremal cases}, and then, following the proof of this very closely,
we consider the non-diagonal extremal case in Section \ref{subsec:s=s'+1}.
In Section~\ref{sec:conclude} we make some brief comments about a recent result of
Ellis, Keller and Lifshitz \cite{EKL2} which is related to Theorem~\ref{thm:weakstability}.

\section{Proofs of Theorems~\ref{thm:inequality} and~\ref{thm:weakstability}}
\label{sec:extFam}

In this section
we derive Theorems~\ref{thm:inequality} and \ref{thm:weakstability} from Theorem~\ref{thm:stability}.
We use some basic asymptotics from Section~4.1.

\subsection{Proof of Theorem~\ref{thm:inequality} (using Theorem~\ref{thm:stability})}

The case $r=0$ of Theorem \ref{thm:inequality} is proved in \cite{FLST}, and hence, we fix $r \geq 1$. 
Let $\epsilon_1$ be a fixed constant in $(0,1/2)$,  and let
$ t_0 $ be $ t_0 (r,\epsilon_1,1/(r+3))$ of Theorem~\ref{thm:stability}. 
Let $n,t,p$ be chosen so that
$n\geq t\geq  t_0 $ and $\frac r{t+2r-1}\leq p\leq\frac{r+1}{t+2r+1}$.

We first suppose that $\cA,\cB\subset 2^{[n]}$ are $t$-nice families.
If condition \eqref{eq:AB_large} of Theorem~\ref{thm:stability} does
not hold, then Theorem~\ref{thm:inequality} is clearly true, so assume
it holds and apply Theorem~\ref{thm:stability}. This gives us
values $(s,s') \in R_{\rm ex}$. 

In the case that $s = s'$, we have $u=v=t$. Item (c) of 
Theorem~\ref{thm:stability} gives
\[ \mu_p(\cA) \mu_p(\cB) \leq \mu_p (\cF^u_s)\mu_p(\cF^v_{s'}) = \mu_p(\cF^t_r)^2, \]
where the equality holds if and only if
$\cA$ and $\cB$ satisfy one of (i), (ii), and (iii) of 
Theorem~\ref{thm:inequality}.

In case that $s-s'=1$, the AM-GM inequality and (a) of 
Theorem~\ref{thm:stability} 
imply that
\begin{align*}\label{eq:AB}
 2\sqrt{\mu_p(\cA) \mu_p(\cB)}&\leq X\leq X_{\cF}.
\end{align*}

The following claim immediately implies~\eqref{eq:main ineq}.
\begin{claim}\label{claim:2opt>pf+g/p}
For $t\geq t_0$ and $(s,s')\in \{(r,r-1),(r+1,r)\}$ we have
$$X_{\cF}=p\mu_p (\cF^u_s)+p^{-1}\mu_p (\cF^v_{s'})<2\mu_p (\cF^t_r).$$
\end{claim}
Except for proving this claim, we have thus proved Theorem~\ref{thm:inequality} provided that
$\cA$ and $\cB$ are $t$-nice families. 
Now suppose that $\cA$ and $\cB$ are 
(not necessarily shifted) inclusion maximal and cross $t$-intersecting
families. Let $\cA'$ and $\cB'$ be $t$-nice families obtained
from $\cA$ and $\cB$ after applying a sequence of shifting operations.
Then, by the fact we have just proved, the families $\cA'$ and $\cB'$
satisfy the inequality \eqref{eq:main ineq} with the equality conditions. 
Since the measure is invariant under shifting operations, we have $\mu_p(\cA)\mu_p(\cB)=\mu_p(\cA')\mu_p(\cB')$, and hence, 
we still have \eqref{eq:main ineq} for $\cA$ and $\cB$. Moreover, it is known
from Lemma~6 in \cite{LST} that if $\cA'=\cB'\cong\cF^x_y$ then 
$\cA=\cB\cong\cF^x_y$. Thus the equality conditions hold for $\cA$ and $\cB$
as well. This completes the proof of Theorem~\ref{thm:inequality}, up to the
proof of Claim~\ref{claim:2opt>pf+g/p}, which we give now. 

\begin{proof}[Proof of Claim~\ref{claim:2opt>pf+g/p}]
Since $s=s'+1$, we have that $u=t-1 \mbox{ and } v=t+1.$

First, we consider the case where $(s,s')=(r,r-1)$. 
Using $p = \Theta(1/t)$, it follows from \eqref{weight cFtr1} (see also \eqref{weight cFtr}) that
\begin{equation}\label{eqn:2det}
\mu_p (\cF^t_r)=\binom{t+2r}{r}p^{t + 2r - r}q^r(1+\Theta(1/t^2)).
\end{equation}
Hence, we can check Claim~\ref{claim:2opt>pf+g/p} by comparing the main terms of 
$p\mu_p (\cF^u_s)+p^{-1}\mu_p (\cF^v_{s'})$ and  $2\mu_p (\cF^t_r)$. That is, it suffices to show 
$p\binom{t+2r-1}r p^{t+r-1}q^r+\frac1p\binom{t+2r-1}{r-1}p^{t+r}q^{r-1}< 2\binom{t+2r}rp^{t+r}q^r$,
that is,
$\binom{t+2r-1}r+\binom{t+2r-1}{r-1}\frac1{pq}< 2\binom{t+2r}r$,
or equivalently,
$p(1-p)>\frac{r}{t+3r}$.
The LHS is minimized, for $p$ satisfying \eqref{def:p range}, at
$p=\frac{r}{t+2r-1}$; and in this case the above
inequality is equivalent to $t>r^2-r+1$.

Next, we consider the case where $(s,s')=(r+1,r)$. In this case it suffices to show that
$\binom{t+2r+1}{r+1}pq+\binom{t+2r+1}{r}  <2\binom{t+2r}r$,
or equivalently,
$p(1-p)<\frac{(r+1) (t+1)}{(t+2r+1)(t+r+1)}$.
The LHS is maximized at $p=\frac{r+1}{t+2r+1}$ and then the above
inequality is equivalent to~$t>r^2-r-1$.
\end{proof}

 Although we use \cite{FLST} for the case $r=0$, we could have proved this case with 
  Theorem~\ref{thm:stability} as well; but it would have required a slightly more complicated statement.  
  The only problem in applying Theorem~\ref{thm:stability} as is, is the condition $\epsilon/t \leq p$
  in the case $r = 0$.  This condition is used only in proving Lemma \ref{lemma:ttj}, which gives us
  item (c) of Theorem~\ref{thm:stability} in the case that $s \neq s'$.
  However, we only need (a) of Theorem~\ref{thm:stability} to prove Theorem~\ref{thm:inequality}.

\subsection{Proof of Theorem~\ref{thm:weakstability} (using Theorem~\ref{thm:stability})}
Let $r,\epsilon$, and $C > 2$ be given, and set $\epsilon_1$ so that
$C=2/(1-2\epsilon_1)$, 
(which implies $\epsilon_1<1/2$).
Let $\delta_1=2\epsilon/(r+2)$, and
$t_0=t_0(r,\epsilon_1,\delta_1)$ for $r\geq 1$
or $t_0=t_0(\epsilon,\epsilon_1,\delta_1)$ for $r=0$
be determined by Theorem~\ref{thm:stability}.
Let $\cA$ and $\cB$ be families, as in the setup of Theorem \ref{thm:weakstability}, that satisfy \eqref{eq:AB_large(2)}.  
Then \eqref{eq:AB_large} holds, and we can apply
Theorem~\ref{thm:stability} to get $(s,s')\in R_{\rm ex}$ for which
(a)--(c) hold.
We consider the following three cases separately:
\emph{Case 1} $(s,s')=(r+1,r)$ or $(r,r-1)$,
\emph{Case 2} $(s,s')=(r+1,r+1)$ or $(r-1,r-1)$,
\emph{Case 3} $(s,s')=(r,r)$.

\noindent $\bullet$ \emph{Case 1.} Since $s=s'+1$, we have
$u=t-1$ and $v=t+1$.
By (c) of Theorem~\ref{thm:stability}, we have
\begin{equation}\label{eq:AB/F2}
\mu_p(\cA)\mu_p(\cB)
\leq \mu_p(\cF^{t-1}_s)\mu_p(\cF^{t+1}_{s'}).
\end{equation}

\begin{claim}\label{claim:for corollary2b}
Let  $(s,s')\in\{(r+1,r),(r,r-1)\}$ and $0<\epsilon<1/2$. 
 If~$\frac{r+\epsilon}t\leq p\leq\frac{r+1-\epsilon}t$, then
\begin{equation}\label{eq:<1-e/(r+2)}
\frac{\mu_p(\cF^{t-1}_{s})\mu_p(\cF^{t+1}_{s'})}{\mu_p(\cF^t_r)^2}\leq
\left(1-\frac{\epsilon}{r+1}\right)(1+o(1)).\end{equation} 
\end{claim}
\begin{proof}
Recall that $u = t-1$ and $v = t+1$.  
First, let $(s,s')=(r+1,r)$. 
Note that $r+\epsilon<tp<r+1-\epsilon$.
Using \eqref{eqn:2det}, and in the last inequality that $tp < r+1-\epsilon$,  we get 
\begin{align*}
\frac{\mu_p(\cF^u_s)\mu_p(\cF^v_{s'})}{\mu_p(\cF^t_r)^2}&=
\frac{\binom{t+2r+1}{r+1}p^{t+r}q^{r+1}\binom{t+2r+1}{r}p^{t+r+1}q^{r}}{\left(\binom{t+2r}rp^{t+r}q^r\right)^2}(1+o(1))\\
&=\frac{(t+2r+1)^2pq}{(r+1)(t+r+1)}(1+o(1))
=\frac{tpq}{r+1}(1+o(1)) \\
&\leq\left(1-\frac{\epsilon}{r+1}\right)(1+o(1)).
\end{align*}

Next, for the case $(s,s')=(r,r-1)$ one can similarly check that  
\begin{align*}
\frac{\mu_p(\cF^u_{s})\mu_p(\cF^v_{s'})}{\mu_p(\cF^t_r)^2}
&=\frac{\binom{t+2r-1}{r}p^{t+r-1}q^{r}\binom{t+2r-1}{r-1}p^{t+r}q^{r-1}}{\left(\binom{t+2r}rp^{t+r}q^r\right)^2}(1+o(1))\\
&=\frac{(t+r)r}{(t+2r)^2pq}(1+o(1))=\frac{r}{tpq}(1+o(1))\\
&\leq\frac r{r+\epsilon}(1+o(1))<1-\frac{\epsilon}{r+1}. \qedhere
\end{align*}
\end{proof}
Combining~\eqref{eq:AB/F2} and~\eqref{eq:<1-e/(r+2)} 
contradicts \eqref{eq:AB_large(2)}, and hence Case 1 cannot happen.

\noindent $\bullet$ \emph{Case 2.}
Part (c) of Theorem~\ref{thm:stability} implies that
\begin{equation}\label{eq:AB/F2(2)}
\mu_p(\cA)\mu_p(\cB)
\leq\mu_p(\cF^{t}_s)^2.
\end{equation}

\begin{claim}\label{claim:for corollary2}
Let $0<\epsilon<1/2$. If $\frac{r+\epsilon}t\leq p\leq\frac{r+1-\epsilon}t$, then
\begin{equation}\label{eq:max F^t_r-1, F^t_r+1}
\frac{\max\{\mu_p(\cF^t_{r-1}),\mu_p(\cF^t_{r+1})\}}{\mu_p(\cF^t_r)}
\leq \left(1-\frac{\epsilon}{r+1}\right) (1+o(1)). 
\end{equation}
\end{claim}
\begin{proof}
Here we have $u = v = t$. By \eqref{eqn:2det} we get 
\begin{align*}
\frac{\mu_p(\cF^t_{r-1})}{\mu_p(\cF^t_r)}&=
\frac{\binom{t+2r-2}{r-1}p^{t+r-1}q^{r-1}}
{\binom{t+2r}rp^{t+r}q^r}(1+o(1))
=\frac{r(t+r)}{(t+2r)(t+2r-1)pq}(1+o(1))\\
&=\frac{r(1+o(1))}{tpq}
\leq \frac{r(1+o(1))}{r+\epsilon}< 
1-\frac{\epsilon}{r+1},\end{align*}
and
\begin{align*}
\frac{\mu_p(\cF^t_{r+1})}{\mu_p(\cF^t_r)}&=
\frac{\binom{t+2r+2}{r+1}p^{t+r+1}q^{r+1}}
{\binom{t+2r}rp^{t+r}q^r}(1+o(1))
= \frac{(t+2r+2)(t+2r+1)pq}{(r+1)(t+r+1)}(1+o(1))\\
&=\frac{tpq(1+o(1))}{r+1}\leq \left(1-\frac{\epsilon}{r+1}\right)(1+o(1)). \qedhere
\end{align*}
\end{proof}
Combining~\eqref{eq:AB/F2(2)} and~\eqref{eq:max F^t_r-1, F^t_r+1} again contradicts \eqref{eq:AB_large(2)}, and hence Case 2 cannot happen.

\noindent $\bullet$ \emph{Case 3.}
By the AM-GM inequality, and \eqref{eq:AB_large(2)}, we have
\[ X\geq 2\sqrt{\mu_p(\cA)\mu_p(\cB)}\geq 2\sqrt{1-\delta}\, \mu(\cF^t_r). \]
Recalling $X,X_{\cF},X_{\Delta},X_{*}$ from the definition preceding Theorem~\ref{thm:stability}, an inclusion-exclusion argument gives that
\begin{align*}
 X_{\Delta} &=X_{\cF}+2X_{*} -X \leq 2\mu_p(\cF^t_r)+2X_{*}-2\sqrt{1-\delta}\mu_p(\cF^t_r) \\
 &=2(1-\sqrt{1-\delta})\mu_p(\cF^t_r)+2X_{*} .
\end{align*}
This together with $X_{*} \leq\epsilon_1 X_{\Delta} $ from (b) of Theorem~\ref{thm:stability} yields
\[X_{\Delta} \leq 2(1-\sqrt{1-\delta})\mu_p(\cF^t_r)+2\epsilon_1 X_{\Delta}.\]

Since $0<\epsilon_1<1/2$ we have
\begin{equation}\label{eq:bndXdelta}
X_{\Delta} \leq\frac{2(1-\sqrt{1-\delta})}{1-2\epsilon_1}\mu_p(\cF^t_r)=C(1-\sqrt{1-\delta})\mu_p(\cF^t_r).
\end{equation}
This proves \eqref{eq:<f(gamma)}, and completes the proof of Theorem~\ref{thm:weakstability}.

\section{Preliminary tools}\label{sec:prelim}

\subsection{Walks corresponding to subsets}
It is useful to regard a set $F\subset [n]$ as an $n$-step walk starting at the
origin $(0,0)$ of the two-dimensional grid $\mathbb{Z}^2$ as follows.
If $i\in F$, then the $i$-th step is \emph{up} from $(x,y)$ to $(x,y+1)$.
Otherwise, the $i$-th step is \emph{right} from $(x,y)$ to $(x+1,y)$.
From now on, we refer to $F\subset [n]$ as a set or a walk. 

\begin{figure}[h]
\begin{center}
\includegraphics{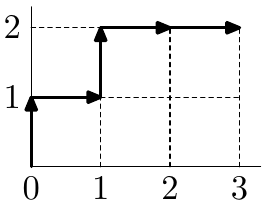}  
\caption{Walk corresponding to $\{1,3\} \subset [5]$}
\end{center}
\end{figure} 

A walk $F$ reaches the point $(j,j+\ell)$ on the line $y = x + \ell$
if and only if $|F\cap[\ell+2j]|=\ell+j$.
 For an integer $\ell \geq 0$ let
\begin{equation*}\label{eq:cFl}
  \cF^{\ell}:= \bigcup_{j \geq 0} \cF_j^\ell =  \Big\{F\subset[n]: \, \big|F\cap[\ell+2j]\big|\geq \ell+j \text{ for some }j\geq 0\Big\}  
\end{equation*}
be the family of all walks that hit the line $y = x + \ell$,
where $\cF^{\ell}_j$ is defined in (1).

We further define
\begin{align*}\tilde\cF^{\ell}&:=\cF^{\ell+1}=\big\{F\in\cF^{\ell}: \, F \text{ hits $y = x + \ell+1$}\big\},\\
\dot\cF^{\ell}&:=\big\{F\in\cF^{\ell}: \, F \text{ hits $y = x + \ell$ exactly once, but does not hit $y=x+\ell+1$}\big\}, \\
\ddot\cF^{\ell}&:=
 \big\{F\in\cF^{\ell}: \, F \text{ hits $y = x + \ell$ at least twice, but does not hit $y=x+\ell+1$}\big\}.
\end{align*}
This gives a partition $\cF^{\ell}=\tilde\cF^{\ell} \sqcup \dot\cF^{\ell} \sqcup\ddot\cF^{\ell}$.
Let $\dot\cF^{\ell}_i:=\dot\cF^{\ell}\cap\cF^{\ell}_i$ and 
$\ddot\cF^{\ell}_i:=\ddot\cF^{\ell}\cap\cF^{\ell}_i$.  

One can estimate the measure of the families defined above using random walks, and this is one of the main ideas for proving our results.
Consider an infinite random walk in the plane starting from the 
origin, each step of which is a random variable, independent of other steps, going up with probability $p$ and right with probability
$q:=1-p$.
The product measure of the family of walks that satisfy some property is
the probability of a random walk satisfying that property.

\begin{example}\label{ex:walkshittingpoint}
 Let $\cF$ be the family of all walks that hit the point $(s, u+s)$ for
 integers $u$ and $s$. There are $\binom{u+2s}{s}$ different walks from
 $(0,0)$ to $(s, u+s)$, and a random walk is any one of these
 with probability $p^{u+s}q^{s}$, so $\mu_p (\cF) = \binom{u+2s}{s}p^{u+s}q^{s}$.
\end{example}

With a little more work, one can show that the infinite random walk
hits the line $y=x+\ell$ with probability precisely $\alpha^{\ell}$, where
$\alpha:=p/q$. Based on this fact,
one can show the following.

\begin{lemma}[Lemma~2.2 in \cite{FLST}]\label{lemma:alpha^l bound}
For any positive integer $\ell$, the following hold.
\begin{enumerate}
\item We have $\mu_p (\cF^{\ell})\leq \alpha^{\ell}$, $\mu_p (\tilde\cF^{\ell})\leq \alpha^{\ell+1}$,
and $\mu_p (\ddot\cF^{\ell})\leq \alpha^{\ell+1}$.
\item If $\cF\subset 2^{[n]}$
and no walk in $\cF$ hits the line $y=x+\ell$, then
$\mu_p (\cF)<1-\alpha^{\ell}+o(1)$, \\ where $o(1)\to 0$ as $n\to\infty$.
\end{enumerate}
\end{lemma}

We will also use the following fact.
\begin{fact}[Lemma 2.13 in~\cite{FLST}]\label{fact:reflection}
  The number of walks from $(0,0)$ to $(s,\ell+s)\in\N^2$ that do not hit the
  line $y=x+\ell+1$ is $\binom{\ell+2s}s-\binom{\ell+2s}{s-1}$.
\end{fact}

\subsection{Dual walks and some facts}

For $A\subset [n]$, let $(A)_i$ be the $i$-th smallest element of $A$. 
For $A,B\subset [n]$, we say that \emph{$A$ shifts to $B$}, denoted by
$A\shiftsto B$, if $|A|\leq |B|$ and $(A)_i\geq (B)_i$ for each $i\leq |A|$. 
As an example, we have $\{2,4,6\}\shiftsto \{1,4,5,7\}$.
Viewing $A$ and $B$ as walks,  $A\shiftsto B$ means that the walk $B$ is above 
the walk $A$ on the grid $\Z^2$.
The \emph{dual} of $A\subset[n]$ with respect to $t$ is defined by 
\begin{equation}\label{eq:dual}
\dual(A):=[(A)_t-1]\cup ([n]\setminus A).     
\end{equation}
See Figure~\ref{fig:dual1} for an example of a walk and its dual. 

Note that $|A \cap \dual(A)| = t-1$. Furthermore, $\dual(A)$ is the shift minimal
walk satisfying this condition, and hence, if $|A\cap B|=t-1$ then
$B\to\dual(A)$. As walks, $\dual(A)$ is obtained by
reflecting $A$ across the line $y=x+(t-1)$ and 
replacing the part $x<0$ with the path connecting $(0,0)$ and
$(0,(A)_t-1)$.
The following fact is immediate.
\begin{fact}[Facts 2.8 and 2.9 in~\cite{FLST}]\label{fact:dual} $\phantom{}$
\begin{itemize}
\item[(i)] Let $\cF$ be a shifted, inclusion maximal family in $2^{[n]}$. 
If $F\in \cF$ and $F\shiftsto F'$, then $F'\in \cF$.
\item[(ii)] Let $\cA$ and $\cB$ be cross $t$-intersecting families. 
If $A\in \cA$, then $\dual(A)\not\in \cB$.
\end{itemize}
\end{fact}

\section{Setup for the proof of Theorem~\ref{thm:stability} }

The rest of the paper is dedicated to the proof of Theorem \ref{thm:stability}.

\subsection{Basic Asymptotics}\label{sec:asymp}
In this section we talk about the assumptions we will use in asymptotic arguments.
From now on, we let $r\geq 0$ be a fixed integer, and if $r=0$ then let $\epsilon > 0$ also be fixed.  Let
let $t$ be a sufficiently large integer depending on $r$ and $epsilon$. 
Since we are interested in the maximum possible measure $\mu_p (\cA)\mu_p (\cB)$
over all $t$-nice pairs $\{\cA,\cB\}$, where $\cA,\cB\subset 2^{[n]}$, 
and this value is non-decreasing in $n$ (see Lemma~2.12 in \cite{FLST}), 
we may assume that $n$ is sufficiently large
compared to $t$. Consequently, we assume that
\[ 0\leq r \ll t\ll n. \]

First assume that $r\geq 1$,
and $p$ satisfies 
\begin{equation*}\label{eq:p}\frac{r}{t+2r-1}\leq p\leq\frac {r+1}{t+2r+1},
\end{equation*}
which gives 
\begin{equation}\label{eq:tp} 
p=\Theta(1/t)=o(1) \;\text{ and }\; r-o(1) < tp < r+1,
\end{equation}
where $o(1)$ goes to $0$ as $t \to \infty$.
This implies that
$\left(1-\frac{r+1}t\right)^t < (1-p)^t < \left(1-\frac{r-o(1)}t\right)^t$, 
and using $q=1-p$ it follows that
$e^{-(r+1)}(1-o(1)) < q^t < e^{-r}(1+o(1))$;
in particular, $ q^t=\Theta(1)$. We also have that
\begin{equation}\label{eq:alpha2}
  \alpha=\frac pq=p+p^2+\cdots=p+O(p^2)  \hskip 1em \mbox{ and } \hskip 1em
  \alpha^t = \frac{p^t}{q^t} = \Theta(p^t).
\end{equation}  

To simplify \eqref{weight cFtr1} we write
\begin{equation}\label{weight cFtr}
  \mu_p (\cF_r^t) = \sum_{i = 0}^r \binom{t+2r}{i}p^{t + 2r - i}q^i= \binom{t+2r}rp^{t+r}q^r (1 + \Theta(1/t^2)) = \frac{(tp)^r}{r!}p^t (1 + \Theta(1/t^2)).
\end{equation} 
Using \eqref{eq:tp}, this reads
\begin{equation}\label{weight cFtr2}
  \mu_p (\cF_r^t) = \Theta(p^t).
\end{equation}
Suppose now that $r=0$ and that $p > \epsilon/t$.  Then \eqref{eq:alpha2} and \eqref{weight cFtr2} trivially hold.
In particular, $\mu_p (\cA)\mu_p (\cB) = \Theta(p^{2t})$
holds for any pair $\{\cA, \cB\}$ in \eqref{eq:extremal families} 
for $r\geq 1$, as well as $\{\cA, \cB\}$ in 
$\{\cF^t_0,\cF^t_0\}$, $\{\cF^t_1,\cF^t_1\}$, and $\{\cF^{t-1}_1,\cF^t_0\}$.

\subsection{Initial reductions and definition of cases}\label{subsec:setup}

In this section, we make initial reductions for the proof of
Theorem~\ref{thm:stability}, and introduce parameters by which we can
break the remainder of the proof into cases.  

We have already fixed $r\geq 0$. Let $\epsilon, \epsilon_1\in(0,1/2)$ and
$\delta_1\in(0,1/(r+2))$ be given. We will choose $t_0$, depending on these
constants, to be sufficiently large. Only $\delta_1$ matters in this and the next sections,
and then $\epsilon$ or $\epsilon_1$ will get involved in 
Sections~\ref{sec:extremal cases} and \ref{subsec:s=s'+1}.
Choose $n\gg t\geq t_0$.

Let $\cA,\cB\subset 2^{[n]}$ be $t$-nice families. 
Where $\lambda(\cF)$, for $\cF\subset 2^{[n]}$, is the maximum $\ell$ 
such that $\cF \subset \cF^\ell$, let
\[ u:= \lambda(\cA) \mbox{ and } v:= \lambda(\cB). \]
Without loss of generality, we may assume that $u \leq v$. 
Since $\cA$ and $\cB$ are $t$-nice,
we have $u+v\geq 2t$, see e.g., Lemma~2.11 of \cite{FLST}.

Since $\cA\subset\cF^u$, we have a partition 
$\cA=\tilde\cA\sqcup\dot\cA\sqcup\ddot\cA$, where
$\tilde\cA=\cA\cap\tilde\cF^u$, $\dot\cA=\cA\cap\dot\cF^u$, and
$\ddot\cA=\cA\cap\ddot\cF^u$. 
Similarly, we have a partition 
$\cB=\tilde\cB\sqcup\dot\cB\sqcup\ddot\cB\subset\cF^v$.

\begin{lemma}\label{lemma:ttd}
If either $u+v>2t$, $\dot\cA=\emptyset$, or $\dot\cB=\emptyset$, then
$\mu_p(\cA)\mu_p(\cB)\ll\mu_p(\cF^t_r)^2$.
\end{lemma}

\begin{proof}
Suppose that $u+v>2t$. Lemma~\ref{lemma:alpha^l bound} yields that
$\mu_p (\cA)\leq \alpha^u$ and $\mu_p (\cB)\leq \alpha^v$,
so by \eqref{eq:alpha2} and \eqref{weight cFtr2} it follows
$\mu_p (\cA)\mu_p (\cB)\leq \alpha^u\alpha^v \leq
\alpha^{2t+1} = O(p^{2t+1})
\ll  \mu_p (\cF^t_r)^2$.

Next suppose that $\dot\cA=\emptyset$, then 
Lemma~\ref{lemma:alpha^l bound} gives 
$\mu_p (\cA)=\mu_p (\tilde\cA)+\mu_p (\ddot\cA)\leq 2\alpha^{u+1}$, and thus
$\mu_p (\cA)\mu_p (\cB)\leq 2\alpha^{u+1}\alpha^v = 2\alpha^{2t+1}\ll\mu_p (\cF^t_r)^2$.
The same holds if $\dot\cB=\emptyset$. 
\end{proof}

Lemma~\ref{lemma:ttd} guarantees the existence of $t_0$ depending on $r$
and $\delta_1$
such that if $t\geq t_0$ and $\cA$ and $\cB$ are $t$-nice families satisfying \eqref{eq:AB_large},
then we necessarily have that $u+v=2t$, $\dot\cA\neq\emptyset$, and 
$\dot\cB\neq\emptyset$.
Moreover, one can show the following.
\begin{lemma}[Lemma~3.2 in \cite{FLST}]\label{lemma:A subset F^u_s} 
Suppose that $\dot\cA\neq\emptyset$ and $\dot\cB\neq\emptyset$. Then,
there exist unique non-negative integers $s$ and $s'$ such that
$\dot\cA\sqcup\ddot\cA\subset\cF_s^u$ and
$\dot\cB\sqcup\ddot\cB\subset\cF_{s'}^v$.
Moreover, $s-s'=(v-u)/2$. In particular, $s\geq s'$.
\end{lemma}

Here, we record the main discussion of this and the previous section. 

\medskip\noindent
{\bf Setup.}
For a proof of Theorem~\ref{thm:stability} we may assume the following.
\begin{itemize}
\item $r\geq 0$ is a fixed integer.
\item $\epsilon,\epsilon_1\in(0,1/2)$, and $\delta_1\in(0,\frac1{r+2})$
are fixed real numbers.
\item $r\ll t\ll n$. 
($t$ depends on $r,\epsilon,\epsilon_1,\delta_1$, which will be described later.)
\item $\frac{r}{t+2r-1}\leq p\leq\frac {r+1}{t+2r+1}$
for $r\geq 1$, and $\frac{\epsilon}t\leq p\leq\frac1{t+1}$ for 
$r=0$.
\item $q = 1-p$ and $\alpha=p/q$.
\item $\cA,\cB\subset 2^{[n]}$ are $t$-nice, that is, shifted,
inclusion maximal, and cross $t$-intersecting.
\item $u=\lambda(\cA)$, $v=\lambda(\cB)$, $u+v=2t$, 
$1\leq u\leq t\leq v\leq 2t$.
\item $s\geq s' \geq 0$, $s-s'=(v-u)/2$, $u=t-(s-s')$, $v=t+s-s'$.
\item $\cA=\tilde\cA\sqcup\dot\cA\sqcup\ddot\cA\subset\cF^u$,
$\cB=\tilde\cB\sqcup\dot\cB\sqcup\ddot\cB\subset\cF^v$,
$\dot\cA\neq\emptyset$, $\dot\cB\neq\emptyset$, 
$\dot\cA\sqcup\ddot\cA\subset\cF_s^u$, 
$\dot\cB\sqcup\ddot\cB\subset\cF_{s'}^v$.
\end{itemize}

\medskip
Under this setup, the proof of Theorem~\ref{thm:stability} breaks down into three cases: Recalling $R_{\rm ex}$ defined in~\eqref{def:R_ex}, we consider
{\bf NE}: the non-extremal case $(s,s') \not\in R_{\rm ex}$, 
{\bf DE}: the diagonal extremal case $(s,s') \in R_{\rm ex}$ and $s = s'$, 
{\bf NDE}: the non-diagonal extremal case $(s,s') \in R_{\rm ex}$ and $s = s'+1$.

\section{Non-extremal cases}\label{sec:ABEC}

In this section, 
we deal with the case NE, that is, the case when $(s,s')\not\in R_{\rm ex}$. We prove the following lemma. 
By choosing $t_0$, depending on $\delta_1$, sufficiently large, 
the lemma shows that Theorem~\ref{thm:stability} holds vacuously for $t\geq t_0$,
since~\eqref{eq:AB_large} does not hold.

\begin{lemma}\label{lemma:AAC}
If $(s,s')\not\in R_{\rm ex}$ then 
$\mu_p (\cA)\mu_p (\cB)<\left(1-\frac1{r+2}\right)\mu_p (\cF^t_r)^2(1+o(1))$.
\end{lemma}

For a proof of Lemma~\ref{lemma:AAC}, recall that $\cA=\tilde\cA\sqcup\dot\cA\sqcup\ddot\cA\subset
\tilde\cF^u\sqcup\dot\cF^u_s\sqcup\ddot\cF^u_s$.
Recall also that no walk in $\dot\cF^u_s\sqcup\ddot\cF^u_s$ hits the line
$y=x+u+1$ while all walks in $\cF^u_s$ hit one of 
$(0,u+2s), (1,u+2s-1), \ldots, (s,u+s)$.
Since $\dot\cF^u_s\sqcup\ddot\cF^u_s\subset\cF^u_s$ all walks in
$\dot\cF^u_s\sqcup\ddot\cF^u_s$ must hit $(s,u+s)$. 
Thus, by Example~\ref{ex:walkshittingpoint},
we have $$\mu_p (\dot\cF^u_s\sqcup\ddot\cF^u_s)\leq 
\binom{u+2s}sp^{u+s}q^s<\binom{u+2s}sp^{u+s}.$$
This together with
$\mu_p (\tilde\cF^u)\leq \alpha^{u+1}=(p/q)^{u+1}$, which we get from 
Lemma~\ref{lemma:alpha^l bound}, yields
\[
\mu_p (\cA)\leq\mu_p (\tilde\cF^u\sqcup\dot\cF^u_s\sqcup\ddot\cF^u_s)< 
\left(\frac pq\right)^{u+1}+\binom{u+2s}sp^{u+s}
=h(u,s)\, p^u,
\]
where
\begin{equation}\label{def:h(u,s)}
 h(i,j):=\frac p{q^{i+1}}+\binom{i+2j}jp^j. 
\end{equation}
Similarly we have $\mu_p (\cB)\leq h(v,s')p^{v}$.
Thus we have $$\mu_p (\cA)\mu_p (\cB)<h(u,v)h(v,s')p^{u+v}=h(u,v)h(v,s')p^{2t}.$$ 
Hence, in order to show Lemma~\ref{lemma:AAC}, it suffices to show
\begin{equation}\label{eq:lemmaAAC}
 h(u,s)h(v,s')p^{2t}<\frac{r+1}{r+2}\mu_p (\cF^t_r)^2(1+o(1)).
\end{equation}

\subsection{The case when $s$ is large} 
In this subsection we show~\eqref{eq:lemmaAAC} for the case $s\geq 2e(r+1)$.
We start by bounding the terms in $h(u,s)$ and $h(v,s')$ using the following claim, which uses only elementary calculus.
\begin{claim}
If $s\geq 2e(r+1)$, then we have
\begin{enumerate}
\item $\max\{p/q^{u+1},p/q^{v+1}\}=O(p)$,
\item $\binom {u+2s}{s} p^s<e^{-e(r+1)}$,
\item $\binom {v+2s'}{s'} p^{s'}<1.001e^{2(r+1)}$.
\end{enumerate}
\end{claim}
\begin{proof} (i) Using $u\leq v\leq 2t$ and $q^t=\Theta(1)$, we get
$\max\{p/q^{u+1},p/q^{v+1}\}\leq p/q^{2t+1}=O(p)$.

(ii) We have
$ \binom{u+2s}sp^s<\frac{(u+2s)^s}{s!}p^s
<\frac{(u+2s)^s}{(s/e)^s}p^s=\left(\frac {eup}s+2ep\right)^s$. 
Using $up\leq tp< r+1$ and the assumption $s\geq 2e(r+1)$,
we have $(eup)/s<  1/2$.
Also, we have $2ep<0.1$. 
Thus we get
$\binom{u+2s}sp^s<\left(\frac12+0.1\right)^s
<\left(\frac1{\sqrt{e}}\right)^s
\leq\left(\frac1{\sqrt{e}}\right)^{2e(r+1)}=e^{-e(r+1)}$.

(iii) If $s'=0$ then the inequality is true. Let $s'>0$.
First observe
\begin{align*}
 \binom{v+2s'}{s'}p^{s'} &< \binom{2t+2s'}{s'}\left(\frac{r+1}t\right)^{s'}
<\frac{(2t+2s')^{s'}}{(s'/e)^{s'}}\left(\frac{r+1}t\right)^{s'}\\
&=\left(\frac{2e(r+1)(s'+t)}{s't}\right)^{s'}=:\gamma.
\end{align*}
We consider now three cases, depending on $s'$. 
\begin{itemize} \item If $s'\geq\sqrt{t}$ then $\frac{s'+t}{s't}\leq\frac{\sqrt{t}+t}{\sqrt{t}t}=
\frac1t+\frac1{\sqrt{t}}<\frac1{2e(r+1)}$ 
for $t\geq t_0$. Thus $\gamma<1$. 
 \item If $2e(r+1)< s'<\sqrt{t}$ then $\frac{2e(r+1)}{s'}<1$, and hence,
$\gamma<\left(\frac{s'+t}t\right)^{s'}<
\left(\frac{t+\sqrt{t}}{t}\right)^{\sqrt{t}}<e$.
\item Finally let $0<s'<2e(r+1)$. We divide $\gamma$ into two parts
$(\frac{s'+t}t)^{s'}$ and $(\frac{2e(r+1)}{s'})^{s'}$.
For the first part we have
$\left(\frac{s'+t}{t}\right)^{s'}<\left(1+\frac{2e(r+1)}{t}\right)^{2e(r+1)}
<1.001$
for $t\geq t_0 $. For the second part, the derivative gives that $(\frac{2e(r+1)}{s'})^{s'}$ is maximized at $s'=2(r+1)$ 
as $e^{2(r+1)}$.
Thus $\gamma<1.001e^{2(r+1)}$.
\end{itemize}
This completes the proof of the claim.
\end{proof}

Now we prove \eqref{eq:lemmaAAC}. By the previous claim we have
\begin{align*}
 h(u,s)h(v,s')&<\left(O(p)+e^{-e(r+1)}\right)
\left(O(p)+1.001e^{2(r+1)}\right)\\
&=1.001e^{-(r+1)(e-2)}+O(p)<1.01e^{-(r+1)(e-2)}.
\end{align*}
Using $\mu_p(\cF_r^t)\geq\mu_p(\cF_0^t)=p^t$ for $r\geq 1$ we get
\[
 \frac{h(u,s)h(v,s')p^{2t}}{\mu_p(\cF^t_r)^2}\leq
h(u,s)h(v,s')<1.01 e^{-2(e-2)}<\frac14<\frac{r+1}{r+2}.
\]
This completes the proof of \eqref{eq:lemmaAAC}, and thus, of 
Lemma~\ref{lemma:AAC} for the case $s\geq 2e(r+1)$.

\subsection{The case when $s$ is small} 
In this subsection we show~\eqref{eq:lemmaAAC} for the case $s< 2e(r+1)$. In this case, we take  advantage of the fact that $s=O(1)$ is much smaller than $t$. Let us estimate $h(u,s)$ defined in~\eqref{def:h(u,s)}.
Using $p/q^{u+1}=o(1)$ and
$u+2s\leq t+2s=t+O(1)$ we have
\[
h(u,s)=\frac p{q^{u+1}}+ \binom{u+2s}sp^s\leq o(1)+ \frac{(t+O(1))^s}{s!}p^s=
\frac{(tp)^s}{s!}(1+o(1)).
\]
Similarly, $h(v,s')=\frac{(tp)^{s'}}{s'!}(1+o(1))$. 
Combining these with $\mu_p(\cF^t_r)=\frac{(tp)^r}{r!}p^t(1+o(1))$
(cf.~\eqref{weight cFtr}), we have 
\begin{equation}\label{eq:g(s,s')}
g(s,s'):=\frac{h(u,s)h(v,s')p^{2t}}{\mu_p(\cF_r^t)^2}=
\frac{(r!)^2(tp)^{s+s'-2r}}{s!s'!}(1+o(1)).
\end{equation}
In order to prove \eqref{eq:lemmaAAC}, it suffices to show that
\begin{equation}\label{eq:g<mu(F)^2}
 g(s,s')<\frac{r+1}{r+2}(1+o(1)).
\end{equation}
Our proof of \eqref{eq:g<mu(F)^2} is based on the following observation.

\begin{claim}\label{claim:(s,s')->(s+1,s'-1)} 
Let $S=\{s\in\N:0\leq s<2e(r+1)\}$.
There is $t_0$ depending on $r$ such that for $t\geq t_0$ the following holds. 
\begin{enumerate}
\item $g(s,s')>g(s+1,s'-1)$ for $s\geq s'>0$.
\item $g(s,s) \leq \max\{g(r-2,r-2), g(r+2,r+2)\}$
for $s\in S\setminus\{r-1,r,r+1\}$,
\item $g(s,s-1)\leq\max\{g(r-1,r-2),g(r+2,r+1)\}$
for $s\in S\setminus\{0,r,r+1\}$.
\end{enumerate}
\end{claim}
Of course, $s$ and $s'$ can never be negative, so when $r = 0$ or $1$ we replace
  $g(\cdot, r-2)$ in (ii) or (iii) with $0$.

\begin{proof}
(i) It follows from~\eqref{eq:g(s,s')} that 
if $t$ is large enough then
$g(s,s')>g(s+1,s'-1)$ is equivalent to $s+1>s'$.

(ii)
Recall that $r-o(1)<tp<r+1$. By comparing $g(x-1,x-1)$ and $g(x,x)$, 
we have that 
$g(1,1) < g(2,2) < \cdots < g(r-1,r-1)$
and $g(r,r) >  g(r+1,r+1) > \cdots$ if $t$ is large enough.

(iii)
Let $F(s):=g(s,s-1)$. Then, by~\eqref{eq:g(s,s')}, we have that 
if $t$ is large enough then $F(s)<F(s+1)$ is equivalent to 
$(s+1)s<(tp)^2$. Hence, 
$F(1)<F(2)<\cdots<F(r)$ and $F(r+1)>F(r+2)>\cdots$. 
Indeed, if $s\leq r-1$, then $(s+1)s\leq r(r-1)<(r-o(1))^2<(tp)^2$.
On the other hand, if $s\geq r+1$, then $(s+1)s\geq (r+2)(r+1)>(r+1)^2>(tp)^2$.
\end{proof}

The arrows in Figure~\ref{figR2} illustrates the relation between the values $g(s,s')$
for values of $s$ and $s'$ considered in Claim~\ref{claim:(s,s')->(s+1,s'-1)}.
We mention that \eqref{eq:g<mu(F)^2} does not hold for 
$(s,s')\in R_{\rm ex}$, which is the reason that
we do not draw arrows starting from the points in $R_{\rm ex}$.
The figure tells us that in order to show 
\eqref{eq:g<mu(F)^2} for $(s,s')\not\in R_{\rm ex}$ it suffices to check
the following starting points when $s$ and $s'$ are non-negative:
\begin{align*}
(s,s')\in \{ &(r-2,r-2), (r+2,r+2), (r-1,r-2), (r+2,r+1), \\
&(r+2,r), (r+2,r-1), (r+1,r-1), (r+1,r-2), (r,r-2)\}.
\end{align*}
\begin{figure}[h]
\begin{center}
\includegraphics[width=7cm]{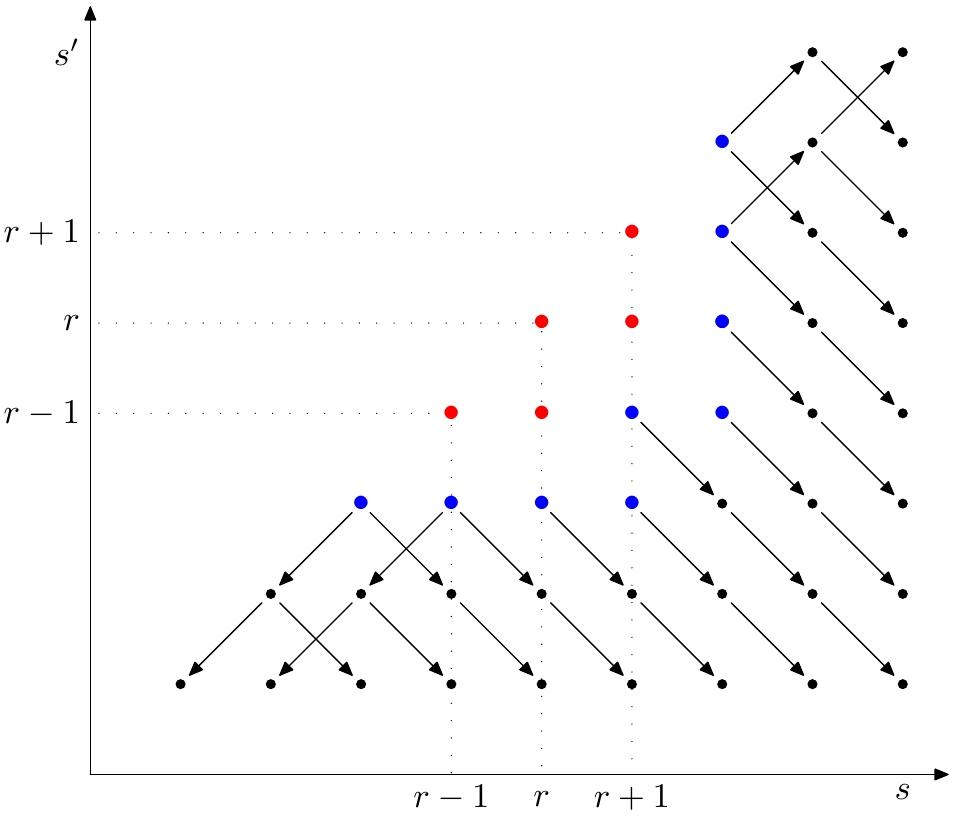}  
\end{center}
\caption{Relations on $g(s,s')$}
\label{figR2}
\end{figure} 

\noindent
The verification of \eqref{eq:g<mu(F)^2} for these cases follows
from easy computation. For example,
\[
g(r-2,r-2) \leq \frac{r^2(r-1)^2}{(tp)^4}(1+o(1))\leq \frac{(r-1)^2}{r^2}(1+o(1))\to \frac{(r-1)^2}{r^2}
\]
as $t\to\infty$, and to mean this situation we write 
$g(r-2,r-2)\nearrow\frac{(r-1)^2}{r^2}$. Similarly we have
\begin{align*}
&g(r+2,r+2)\nearrow\frac{(r+1)^2}{(r+2)^2},\,\,
&&g(r-1,r-2)\nearrow\frac{r-1}r,\,\,
&&g(r+2,r+1)\nearrow\frac{r+1}{r+2},\\  
&g(r+2,r)\nearrow\frac{r+1}{r+2},\,\,
&&g(r+2,r-1)\nearrow\frac{r}{r+2},\,\, 
&&g(r+1,r-1)\nearrow\frac{r}{r+1},\\ 
&g(r+1,r-2)\nearrow\frac{r-1}{r+1},\,\,  
&&g(r,r-2)\}\nearrow\frac{r-1}{r}.
\end{align*}
Therefore, we get \eqref{eq:g<mu(F)^2} for all $(s,s')\not\in R_{\rm ex}$,
which completes the proof of Lemma~\ref{lemma:AAC} for the case $s<2e(r+1)$.

So far, we have proved Theorem~\ref{thm:stability} in the case ND.

\section{Diagonal extremal cases}\label{sec:extremal cases} 
In this section we deal with the case DE, that is, we assume that
\[ (s,s')=(r-1,r-1), (r,r), \mbox{ or } (r+1,r+1),\]
 or the latter two only, if $r = 0$. 
Under the assumption $s=s'$ we have that $u=v=t$. 
Defining the notation
$[a,b]_2:=\{a+2i:i\in\Z\}\cap[a,b]$, 
we let
\begin{align*}
D^t_s(i):=[t-1]\cup [t+s, t+2s] \cup [t+2s+i+2,n]_2\in \dot\cF_s^{t}.
\end{align*}
The parameter $i$ ranges over $1\leq i\leq n-t-2s-1=:i_{\max}$, where $i_{\max}$ 
is defined so that $D^t_s(i_{\max})=[t-1]\cup[t+s,t+2s]$. Recalling the definition of the dual walk 
in~\eqref{eq:dual} on page \pageref{eq:dual}, we have (see Figure~\ref{fig:dual1}) 
\begin{equation*}\label{eq:dual(D)}
  \dual(D^t_s(i))=[t+s-1]\cup[t+2s+1,t+2s+i+1]\cup[t+2s+i+3,n]_2
\end{equation*}

\begin{figure}
\begin{minipage}{7cm}
\includegraphics[scale=.7]{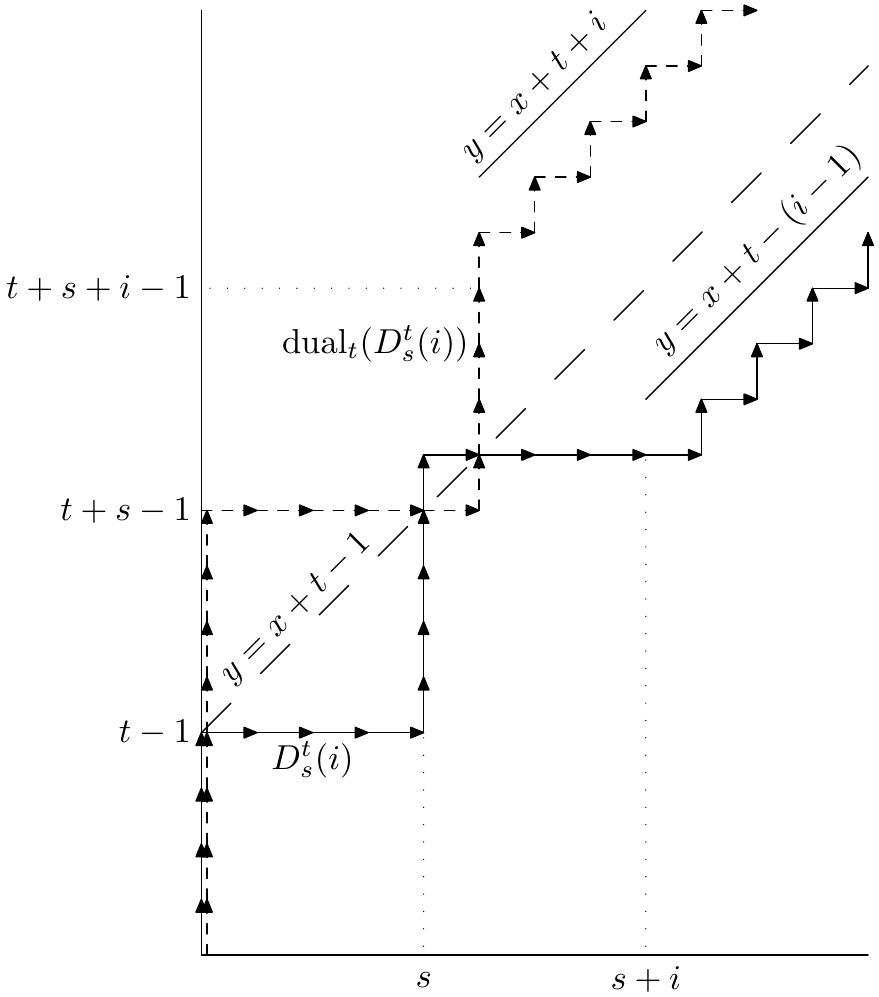} \qquad 
\caption{Walks $D^t_s(i)$ and $\dual(D^t_s(i))$}
\label{fig:dual1}
\end{minipage}
\begin{minipage}{7cm}
\includegraphics[scale=.7]{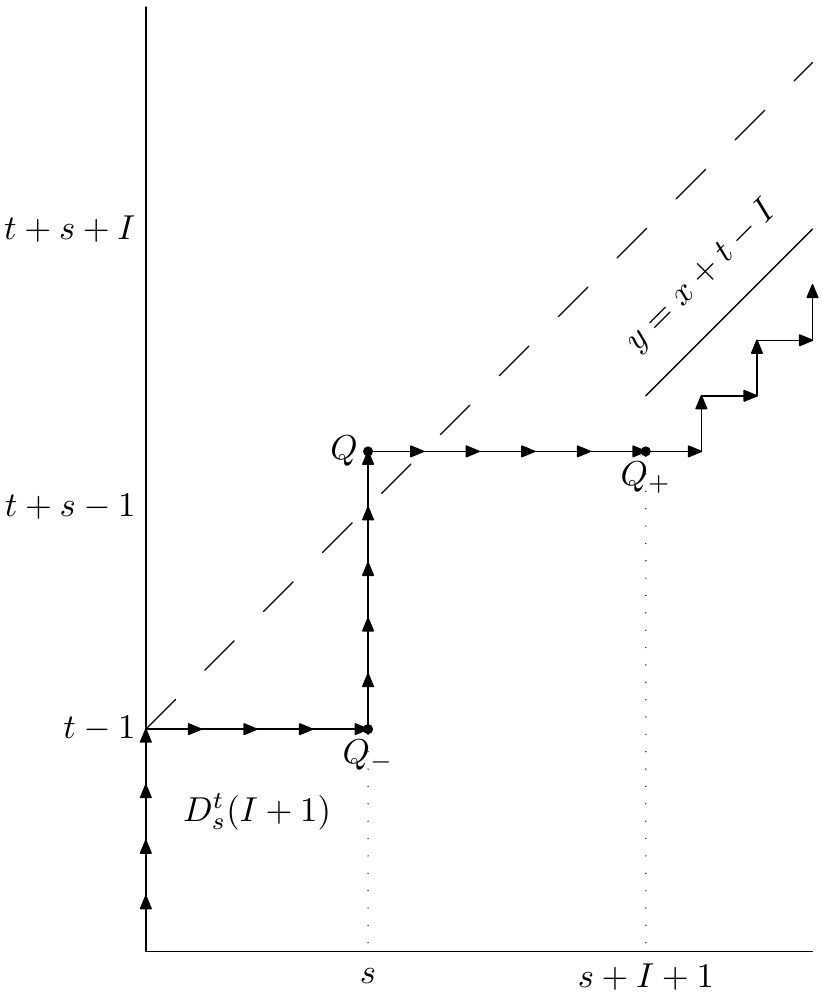}  
\caption{Walk $D^t_s(I+1)$}
\label{fig:QsWalk}
\end{minipage}
\end{figure} 

Consider the case where $D^t_s(1)\not\in \cA$ or $D^t_s(1)\not\in \cB$.
By symmetry we may assume that $D^t_s(1)\not\in \cA$.
In this case we show that Theorem~\ref{thm:stability} vacuously holds 
since~\eqref{eq:AB_large} does not hold.

\begin{lemma}\label{lem:D-tech-CaseIII} 
If $D^t_s(1)\not\in \cA$ then $\mu_p (\cA)\mu_p (\cB) \ll \mu_p (\cF^t_r)^2$, that is,
there exists $t_0$ depending on $r$ and $\delta_1$ such that
$\mu_p (\cA)\mu_p (\cB) < (1-\delta_1)\mu_p (\cF^t_r)^2$
 for all $t\geq t_0$.
\end{lemma}
\begin{proof} 
First we show that
\begin{equation}\label{eq:A(10)}\mu_p (\dot\cA) = O(p^{t+1}).\end{equation}
For the proof, let
$\cW:=\{W\in\dot\cF^t_s:W\to D^t_s(1)\}$. 
Since $D^t_s(1)\not\in \cA$, Fact~\ref{fact:dual} (i) gives that $W\to D^t_s(1)$ implies $W\not\in\cA$.
As $\dot\cA \subset\cF^t_s$ and so $\dot\cA\subset\dot\cF^t_s$ we have 
$\dot\cA\subset\dot\cF^t_s\setminus\cW$, and hence, $\mu_p (\dot\cA)\leq\mu_p (\dot\cF^t_s\setminus\cW)$. 
Now, all walks in $\dot\cF^t_s$ hit the line $y=x+t$ only at $(s,t+s)$.
Hence, they necessarily hit $(s,t+s-1)$ and the $(t+2s)$-th step is `up'.
Using Fact~\ref{fact:reflection} with $\ell = t-1$, the number of ways for a walk in
$\dot\cF^t_s$ to hit $(s,t+s-1)$, and so $(s,t+s)$, is $\binom{t+2s-1}s-\binom{t+2s-1}{s-1}$.
Further, of such walks, those that hit $(s,t-1)$ are in $\cW$, and this can happen in
$\binom{t+s-1}s$ ways.  Therefore, looking at only the first $t+2s$ steps, we 
see that 
\begin{align}
\mu_p (\dot\cA)&\leq\mu_p (\dot\cF^t_s\setminus\cW)\leq \left(\binom{t+2s-1}s-\binom{t+2s-1}{s-1}-\binom{t+s-1}s\right)p^{t+s}q^{s} \nonumber\\
&=\Theta(t^{s-1}p^{t+s}q^s)=\Theta((tp)^{s-1}p^{t+1})=O(p^{t+1}),
\label{eq:mu (A) = O(p^(t+1))}
\end{align}
which gives~\eqref{eq:A(10)}.

Then we use
$\mu_p (\ddot\cA)+\mu_p (\tilde\cA)\leq \alpha^{t+1} =O(p^{t+1})$, and hence, we infer
$\mu_p (\cA)=\mu_p (\dot\cA)+\mu_p (\ddot\cA)+\mu_p (\tilde\cA)=O(p^{t+1})$.  
Since $\mu_p (\cB)\leq \alpha^v=O(p^{t})$,
 we have that
 $\mu_p (\cA)\mu_p (\cB)= O(p^{2t+1})$. On the other hand, \eqref{weight cFtr2} gives 
 $\mu_p (\cF^t_r)^2=\Theta(p^{2t})$.  
 Hence,
$\mu_p (\cA)\mu_p (\cB) \ll \mu_p (\cF^t_r)^2$.
\end{proof}

Now we assume that both $\cA$ and $\cB$ contain $D^t_s(1)$.
Then we can define parameters $I,J$ as follows.
Since $\cA$ is shifted there exists $I$ with $1\leq I\leq i_{\max}$ such that 
$D^t_s(i)\in\cA$ for $i\leq I$ and $D^t_s(i)\not\in\cA$ for $i>I$.
Similarly there is $1\leq J\leq i_{\max}$ such that 
$D^t_s(j)\in\cB$ for $j\leq J$ and $D^t_s(j)\not\in\cB$ for $j>J$.
Based on their values, we consider the following two cases,
\textbf{Case I}: $I=J=i_{\max}$, and
\textbf{Case II}: Either $I\neq i_{\max}$ or $J\neq i_{\max}$.

\medskip
\noindent{\bf Case I.}
In this case we apply the following to get $\cA, \cB \subseteq \cF^t_s$.

\begin{claim}\label{claim:B subset F}
If $I=i_{\max}$, then $\cB \subset \cF^t_s$. If $J=i_{\max}$, then $\cA \subset \cF^t_s$.
\end{claim}
\begin{proof} By symmetry, it is enough to prove only the first statement.
As $I=i_{\max}$, that is, $D^t_s(i_{\max})\in \cA$,  Fact~\ref{fact:dual} (ii) gives that
$\dual(D^t_s(i_{\max})) = [n]\setminus[t+s,t+2s]\not\in\cB$.
So   Fact~\ref{fact:dual} (i) implies that each walk $B\in\cB$ satisfies $B\not\shiftsto \dual(D^t_s(i_{\max})).$ Note that $\dual(D^t_s(i_{\max})) = [t+s-1]\cup[t+2s+1,n],$ which
consists of line segments connecting $(0,0)$, $(0,t+s-1)$, $(s+1,t+s-1)$,
and $(s+1,n-s-1)$.
Thus each walk $B\in\cB$ must hit one of $(0,t+s),(1,t+s),\ldots,(s,t+s)$,
which means $|B\cap[t+2s]|\geq t+s$.
Hence, $\cB\subset\cF_s^t$ holds.
\end{proof}  

One can easily check that Theorem~\ref{thm:stability} (a)--(c) follow from Claim~\ref{claim:B subset F}. Note that the equality in (a) and (c) holds if and only if $\cA=\cB=\cF^t_s$.

\medskip
\noindent{\bf Case II.} First we prove
\begin{equation}\label{X*<XD}
 X_*\ll X_{\Delta}-X_*,
\end{equation}
that is, (b) of Theorem~\ref{thm:stability} holds for $t\geq t_0$,
where $t_0$ depends on $r$ and $\epsilon_1$.
(Note that we do not assume \eqref{eq:AB_large} here.)
Since
$X_*=\mu_p(\cA\setminus\cF^t_s)+\mu_p(\cB\setminus\cF^t_s)$ and 
$X_{\Delta}-X_*=\mu_p(\cF^t_s\setminus\cA)+\mu_p(\cF^t_s\setminus\cB)$,
it suffices to show that
\begin{align}
 \mu_p(\cB\setminus\cF^t_s)&\ll\mu_p(\cF^t_s\setminus\cA),\label{FA>BF}\\
 \mu_p(\cA\setminus\cF^t_s)&\ll\mu_p(\cF^t_s\setminus\cB).\nonumber
\end{align}

By symmetry we only show \eqref{FA>BF}.
If $I=i_{\max}$ then $\cB\subset\cF^t_s$ by Claim~\ref{claim:B subset F}
and \eqref{FA>BF} holds. So we may assume that $I\neq i_{\max}$,
and there exists $D^t_s(I+1)\in\dot\cF^t_s\setminus\cA$.

\begin{claim}
We have
$\mu_p(\cF^t_s\setminus\cA)\geq\Theta(p^tq^I)$ and
$\mu_p(\cB\setminus\cF^t_s)\leq(q/p)^{t+I}$.
\end{claim}

\begin{proof}
To prove the first inequality,
consider walks $W$ such that $W \shiftsto D^t_s(I+1)$ and $W$ hits $Q:=(s,t+s)$.
Since $D^t_s(I+1)\not\in \cA$, we have $W\in\cF^t_s \setminus \cA$ 
by Fact~\ref{fact:dual} (i).
Since $D^t_s(I+1)$ contains line segments connecting
$Q_-:=(s,t-1)$, $Q$, and $Q_{+}:=(s+I+1,t+s)$
it follows that $W$ must hit $Q_-$ and $Q_+$.
(See Figure~\ref{fig:QsWalk}.)
The number of walks from $(0,0)$ to $Q_-$ is $\binom{t+s-1}{s}$, then there is the unique walk passing $Q_-$, $Q$, and $Q_+$ which hits $(s,t+s)$.
So the measure of the family of all such walks $W$ is $\binom{t+s-1}{s}p^{t+s}q^{s+I+1}$.     
After hitting $Q_+$, to satisfy $W\shiftsto D^t_s(I+1)$, walks $W$
must not hit the line $y = x + (t - I)$.
This happens with probability at least $1 - \alpha$
because the measure of walks starting from $Q_+$ which hit this line is
at most $\alpha$ by (i) of Lemma~\ref{lemma:alpha^l bound}.
Thus we obtain 
\begin{align}\label{eq:F-A}
\mu_p (\cF^t_s\setminus \cA)\geq \binom{t+s-1}{s}p^{t+s}q^{s+I+1}(1-\alpha)
=\Theta(t^s p^{t+s} q^I) =\Theta(p^t q^I). 
\end{align}

Next, we prove the second inequality.
Since $D^t_s(I)\in\cA$, we have that $\dual(D^t_s(I))\not\in\cB$.
Referring to Figure \ref{fig:dual1}, with $i = I$, one sees that  
the walk $\dual(D^t_s(I))$ contains line segments connecting
$(0,0)$, $(0,t+s-1)$, $(s+1,t+s-1)$, and $(s+1,t+s+I)$; and then from 
$(s+1,s+t+I)$ the walk never hits the line $y=x+(t+I)$.
Hence, by  Fact~\ref{fact:dual}, each walk $B\in\cB$ must hit one of $(0,t+s), (1,t+s),\dots,(s,t+s)$,
or $y = x+(t+I)$. 
Note that all walks hitting one of these $s+1$ points are contained in 
$\cF^t_s$. Thus, each walk $B\in\cB\setminus \cF^t_s$ hits $y=x+(t+I)$. 
So by Lemma~\ref{lemma:alpha^l bound} (i) we have
\begin{align}\label{eq:B-F}
\mu_p (\cB\setminus \cF^t_s)\leq \alpha^{t+I}=(p/q)^{t+I},
\end{align}
which completes the proof of the claim.
\end{proof}

It follows from the claim that
$\frac{\mu_p(\cF^t_r\setminus\cA)}{\mu_p(\cB\setminus\cF^t_r)}
\geq\Theta\left(\frac{p^tq^I}{(p/q)^{t+I}}\right)=\Theta((q^2/p)^I)\gg 1$,
which yields \eqref{FA>BF}, and then \eqref{X*<XD}. Therefore we have
\begin{align*}
X&=\mu_p(\cA)+\mu_p(\cB)=X_{*} +\mu_p(\cA\cap\cF^t_s)+\mu_p(\cB\cap\cF^t_s)\\
&<X_{\Delta} -X_{*} +\mu_p(\cA\cap\cF^t_s)+\mu_p(\cB\cap\cF^t_s)\leq 
2\mu_p(\cF^t_s)=X_{\cF}, 
\end{align*}
and so we get (a) of Theorem~\ref{thm:stability} without equality. 
Then (c) follows from (a) with the AM-GM inequality, that is,
$\mu_p(\cA)\mu_p(\cB)\leq(X/2)^2<(X_{\cF}/2)^2=\mu_p(\cF^t_s)^2$.

\section{Non-diagonal extremal cases}\label{subsec:s=s'+1}
We now deal with the remaining non-diagonal extremal cases $(s,s')=(r+1,r)$ or $(r,r-1)$. In these cases, we have
$s-s'=1$, and hence, $u=t-1$, and $v=t+1$.
The argument here is similar to that of Section \ref{sec:extremal cases}, but it 
is slightly complicated because of the asymmetry
of $u$ and $v$.

As in Section~\ref{sec:extremal cases}, let
\[
 D^{u}_s(i) =[u-1]\cup[u+s,u+2s]\cup[u+2s+i+2,n]_2\in\dot\cF^{u}_s,
\]
for $1\leq i\leq n-u-2s-1=:i_{\max}$, where $i_{\max}$ is defined so that
$D^{u}_s(i_{\max})=[u-1]\cup[u+s,u+2s]$.
It is not hard to check 
\[  \dual(D^{u}_s(i))=[u+s]\cup[u+2s+1,u+2s+i+1]\cup[u+2s+i+3,n]_2 \]  
In particular, $\dual(D^{u}_s(i_{\max}))=[v+s'-1]\cup[v+2s'+1,n]
= [n] \setminus [v+s',v+2s']$.
On the other hand, $\dual(D^v_{s'}(i_{\max}))\neq[n]\setminus[u+s,u+2s]$
because of the asymmetry of $u$ and $v$. 
To fix this problem we vary from Section~\ref{sec:extremal cases} a bit, 
and let  
\[
 \widetilde{D}^{v}_{s'}(j)
 :=[v-2]\cup[v+s'-1,v+2s']\cup[v+2s'+j+2,n]_2\in\dot\cF^{v}_{s'},
\]
for $1\leq j\leq n-v-2s'-1=n-u-2s-1=i_{\max}.$ Then,
\begin{align*}
 \dual(\widetilde{D}^{v}_{s'}(j))&=[v+s'-2]\cup[v+2s'+1,v+2s'+j+1]\cup[v+2s'+j+3,n]_2 \\
 &=[u+s-1]\cup[u+2s+1,u+2s+j+1]\cup[u+2s+j+3,n]_2,
\end{align*}
and $\dual(\widetilde{D}^{v}_{s'}(i_{\max})) = [n] \setminus [u+s, u+2s]$, as desired.
We use $D^u_s(i)$ for $\cA$ and $\widetilde{D}^v_{s'}(j)$ for $\cB$.

First we show that if
$D^{u}_s(1)\not\in \cA$ or $\widetilde{D}^{v}_{s'}(1)\not\in \cB$, then
Theorem~\ref{thm:stability} vacuously holds since
\eqref{eq:AB_large} does not hold.

\begin{lemma}
If $D^u_s(1)\not\in\cA$ or $\widetilde{D}^{v}_{s'}(1)\not\in \cB$, then
$\mu_p (\cA)\mu_p (\cB) \ll \mu_p (\cF^t_r)^2$, that is,
there exists $t_0$ depending on $r$ and $\delta_1$ such that
$\mu_p (\cA)\mu_p (\cB)<(1-\delta_1)\mu_p (\cF^t_r)^2$
 for all $t\geq t_0$.
\end{lemma}
\begin{proof}
If $D^u_s(1)\not\in\cA$ then
replacing $t$ with $u$ in \eqref{eq:mu (A) = O(p^(t+1))} gives
  \begin{align*}
 \mu_p (\dot\cA)&\leq
  \left(\binom{u+2s-1}s-\binom{u+2s-1}{s-1}
     -\binom{u+s-1}s\right)p^{u+s}q^{s}
=O(p^{u+1}) = O(p^t).
  \end{align*}
This together with $\mu_p(\ddot\cA)+\mu_p(\tilde\cA)\leq\alpha^{u+1}=O(p^t)$
yields $\mu_p(\cA)=O(p^t)$.
For $\cB$ we use the obvious estimation $\mu_p(\cB)\leq\alpha^v=O(p^{t+1})$.
Thus $\mu_p(\cA) \mu_p(\cB) =O(p^{2t+1})$.

If $\widetilde{D}^{v}_{s'}(1)\not\in \cB$ then we have $\mu_p(\dot\cB)=O(p^{t+2})$. Indeed, 
with only obvious changes to the proof of
  Lemma \ref{lem:D-tech-CaseIII}, we get that 
  $\dot\cB \subset \dot\cF^v_{s'} \setminus \cW'$ where 
$\cW':=\{W\in\dot\cF^v_{s'}:W\to \widetilde{D}^v_{s'}(1),\,W\text{ hits }(s',v+s')\}$, 
and, using $\mu_p (\dot\cB)\leq\mu_p(\dot\cF^v_{s'}\setminus\cW')$, we have that
  \begin{align*}
 \mu_p (\dot\cB) &\leq
  \left(\binom{v+2s'-1}{s'}-\binom{v+2s'-1}{s'-1}
     -\binom{v+s'-1}{s'}\right)p^{v+s'}q^{s'} \\
                 &=O(p^{v+1}) = O(p^{t+2}).
  \end{align*}                 
  (If $s' = 0$ then $\dot\cF^v_{s'}=\cW'$ and $\mu_p(\dot\cB)=0$.)
Thus $\mu_p(\cB)=O(p^{t+2})$ follows from 
$\mu_p(\ddot\cB)+\mu_p(\tilde\cB)\leq\alpha^{v+1}=O(p^{t+2})$.
Since $\mu_p(\cA)\leq\alpha^u=O(p^{t-1})$ we get
$\mu_p(\cA) \mu_p(\cB) =O(p^{2t+1})$.

Consequently if $D^u_s(1)\not\in\cA$ or $\tilde D^v_{s'}\not\in\cB$ then
$\mu_p(\cA) \mu_p(\cB) =O(p^{2t+1})\ll\mu_p(\cF^t_r)^2$.
\end{proof}  

Now we assume that both $\cA$ and $\cB$ contain $D^t_s(1)$.
We define the parameters 
$I:=\max\{i\geq 1 : D^{u}_s(i)\in \cA \}$ and
$J:=\max\{j\geq 1 : \widetilde{D}^{v}_{s'}(j)\in \cB \}$.
We again have the following two cases, that is,
\textbf{Case I}: $I=J=i_{\max}$, and
\textbf{Case II}: Either $I\neq i_{\max}$ or $J\neq i_{\max}$.

\medskip
\noindent{\bf Case I.} 
In this case we have
$D^{u}_s(i_{\max})\in \cA$ and  $\widetilde{D}^{v}_{s'}(i_{\max})\in \cB$.
We can argue as in the previous section.
The assumption $D^{u}_s(i_{\max})\in \cA$ implies that 
$\dual(D^u_s(i_{\max}))=[n]\setminus[v+s',v+2s']\not\in\cB$, and hence, $\cB \subset \cF^v_{s'}.$
Similarly $\widetilde{D}^{v}_{s'}(i_{\max})\in \cB$ implies 
$\cA \subset \cF^u_{s}.$
One can easily check that Theorem~\ref{thm:stability} (a)--(c) follow from $\cB \subset \cF^v_{s'}$ and $\cA \subset \cF^u_{s}$.

\medskip
\noindent{\bf Case II.} 
First we prove
\begin{equation}\label{X*<XD2}
 X_*\ll X_{\Delta}-X_*,
\end{equation}
that is, (b) of Theorem~\ref{thm:stability} holds for $t\geq t_0$,
where $t_0$ depends on $r$ and $\epsilon_1$.
To this end it suffices to show the following.
\begin{claim}\label{lem:ND-tech-CaseII}
We have $\frac1p\mu_p (\cB\setminus\cF^v_{s'})\ll p\mu_p (\cF^u_s\setminus\cA)$ and
$p\mu_p (\cA\setminus\cF^u_{s})\ll \frac1p\mu_p (\cF^v_{s'}\setminus\cB)$.      
\end{claim}
\begin{proof}
First, we show the first inequality. If $I=i_{\max}$ then this is true
because $\cB\subset\cF^v_{s'}$ and $\mu_p(\cB\setminus\cF^v_{s'})=0$
follow from the argument in Case I. So assume that $I \neq i_{\max}$. 
Changing $t$ to $u$ in the argument used to give \eqref{eq:F-A} 
yields
\[
\mu_p (\cF^u_s\setminus\cA)\geq\binom{u+s-1}sp^{u+s}q^{s+I+1}(1-\alpha) 
\geq\Theta(u^sp^{u+s}q^I)=\Theta(p^{t-1}q^I),
\]
and changing $t$ to $v$ in the argument for \eqref{eq:B-F} gives
$\mu_p (\cB\setminus\cF^v_{s'})\leq\alpha^{v+I}=(p/q)^{t+I+1}$.
Thus 
\[
\frac{p\mu_p(\cF^u_s\setminus \cA)}{\frac1p\mu_p (\cB\setminus\cF^v_{s'})}
\geq\Theta\left(\frac{p^{t-1}q^I}{(p/q)^{t+I+1}}\right)
=\Theta((q^2/p)^I p^{-2})\gg 1,
\]
showing the first inequality.

Next, we prove the second inequality. This is clearly true if $J=i_{\max}$.
So assume that $J \neq i_{\max}$. 
We have 
\[ \widetilde{D}^v_{s'}(J+1)=[v-2] \cup [v+s'-1,v+2s'] \cup [v + 2s' + J + 2,n]_2 \not\in\cB, \]
and walks which hit $(s',v+s')$ 
and shift to
$\widetilde{D}^v_{s'}(J+1)$ belong to $\cF^v_{s'}\setminus\cB$. 
These walks must hit $(s',v-2)$, then $(s',v+s')$ and then $Q_+=(s'+J+1,v+s')$, and
after hitting $Q_+$ they never hit the line $y=x+(v-J)$. Thus
   \[ \mu_p (\cF^v_{s'}\setminus\cB)\geq\binom{v+s'-2}{s'}p^{v+s'}q^{s'+J+1}(1-\alpha) 
\geq\Theta(v^{s'}p^{v+s'}q^J)=\Theta(p^{t+1}q^J),
\]
   while as
   $\dual(\widetilde{D}_{s'}^v(J))= [u+s-1] \cup [u + 2s + 1,, u + 2s + J + 1] \cup [u + 2s + J + 3,n]_2 
   \not\in \cA$
 implies, following the
   argument for \eqref{eq:B-F}, that each walk in
   $\cA \setminus \cF^u_s$ hits $y = x + (u + J)$. Together this gives  
$\mu_p (\cA\setminus\cF^u_{s})\leq\alpha^{u+J}=(p/q)^{t-1+J}$. So
\[
 \frac{\frac1p\mu_p (\cF^v_{s'}\setminus\cB)}{  p\mu_p (\cA\setminus\cF^u_{s}) }
\geq\Theta\left(\frac{p^{t+1}q^J}{(p/q)^{t-1+J}}\right)
=\Theta((q^2/p)^Jp^{-2})
\gg 1,
\]
which proves the second inequality.
 \end{proof}  
 
Next we check that (a) of Theorem~\ref{thm:stability} holds without equality.
Indeed it follows from \eqref{X*<XD2} that 
\begin{align*}
 X&=p\mu_p(\cA)+\frac1p\mu_p(\cB)=
X_{*} +p\mu_p(\cA\cap\cF^t_s)+\frac1p\mu_p(\cB\cap\cF^t_s)\\
&< X_{\Delta} -X_{*} +p\mu_p(\cA\cap\cF^t_s)+\frac1p\mu_p(\cB\cap\cF^t_s)\leq X_{\cF}.
\end{align*} 

Finally we verify (c) of Theorem~\ref{thm:stability}.

\begin{lemma}\label{lemma:ttj}
We have $\mu_p(\cA) \mu_p(\cB)< \mu_p(\cF^u_s)\mu_p(\cF^v_{s'})$
for all $t\geq t_0$, where $t_0$ depends on $r,\epsilon_1,\epsilon$.
\end{lemma}
\begin{proof}
  Let
  \newcommand{\f}{f}
  \newcommand{\g}{g}
  \newcommand{\fm}{f_*}
  \newcommand{\gm}{g_*}
  \newcommand{\am}{a_*}
  \newcommand{\bm}{b_*}
\begin{align*}
\f&=p\mu_p(\cF^u_s),  & \fm&=p\mu_p(\cF^u_s\setminus\cA), & \am&=p\mu_p(\cA\setminus\cF^u_s), \\
\g&=\mu_p(\cF^v_{s'})/p, &\gm&=\mu_p(\cF^v_{s'}\setminus\cB)/p,
 &\bm&=\mu_p(\cB\setminus\cF^v_{s'})/p.
\end{align*}
Let $\beta=\epsilon/2$ for $r=0$, and let
$\beta$ be any fixed constant satisfying $0<\beta<\frac r{r+1}$ 
for $r \geq 1$.
Claim~\ref{lem:ND-tech-CaseII} gives that $\bm<\beta\fm$ and 
$\am<\beta\gm$. So
\begin{align*}
p\mu_p(\cA)=\f-\fm+\am<\f-\fm+\beta\gm,\quad
\mu_p(\cB)/p=\g-\gm+\bm<\g-\gm+\beta\fm,
\end{align*}
and hence, 
$\mu_p(\cA)\mu_p(\cB)<(\f-\fm+\beta\gm)(\g-\gm+\beta\fm)$.
Therefore, in order to show Lemma~\ref{lemma:ttj}, it suffices to show that
\begin{equation}\label{eq:(f-a)(\g-b)<f\g}
 (\f-(\fm-\beta\gm))(\g-(\gm-\beta\fm))\leq \f\g. 
\end{equation}

If both $\fm-\beta\gm$ and $\gm-\beta\fm$ are non-negative, then we clearly
have \eqref{eq:(f-a)(\g-b)<f\g}.
If both of them are negative or $0$, then $\fm\leq \beta\gm\leq 
\beta^2\fm$. This implies $\fm=\gm=0$, and hence, we also get
\eqref{eq:(f-a)(\g-b)<f\g}. Thus we may assume that 
$(f_*-\beta g_*)(g_*-\beta f_*)< 0$, and
by symmetry, $f_*-\beta g_* < 0 < g_*-\beta f_*$.
Since $0<\beta<1$ we have $0<\beta f_*<\beta^{-1}f_*$ and
$g_*-\beta f_* > g_*-\beta^{-1}f_*$. 
Also we have
$0<\beta g_*-f_*=\beta(g_*-\beta^{-1}f_*)<g_*-\beta^{-1}f_*$.
Consequently we have $g_*-\beta f_* > g_*-\beta^{-1}f_*>0$ and 
\begin{equation}\label{eq:<beta}
 \frac{\beta\gm-\fm}{\gm-\beta\fm}
< \frac{\beta\gm-\fm}{\gm-\beta^{-1}\fm}=\beta.
\end{equation}

Since \eqref{eq:(f-a)(\g-b)<f\g} is rewritten as
$(f_*-\beta g_*)(g_*-\beta f_*)\leq f(g_*-\beta f_*)-g(\beta g_*-f_*)$ 
and we know that the LHS is negative, 
in order to prove \eqref{eq:(f-a)(\g-b)<f\g}
it suffices to show that the RHS is non-negative, that is,
\begin{equation}\label{f(g-bf)>g(bg-f)}
f(g_*-\beta f_*)\geq g(\beta g_*-f_*).
\end{equation}

On the other hand, \eqref{weight cFtr} implies that
\begin{align*}
\frac \f\g=
\frac{\binom{u+2s}sp^{u+s}q^s}{\binom{v+2s'}{s'}p^{v+s'}q^{s'}}p^2
(1+o(1))
=\frac{(u+2s)!(v+s')!s'!}{(v+2s')!(u+s)!s!}pq(1+o(1))
=\frac{t+s}{s}pq(1+o(1)).
\end{align*}
If $r\geq 1$, then using $u=t-1$, $v=t+1$, $s'=s-1$, 
and noting that $pq$ attains its minimum at $p=\frac r{t+2r-1}$,
we get
\[
 \frac fg=\frac{t+s}spq(1+o(1))
\geq \frac{t+r+1}{r+1}\frac r{t+2r-1}\frac{t+r-1}{t+2r-1}(1+o(1))
=\frac r{r+1}(1+o(1))>\beta.
\]
Similarly, if $r=0$, then $(s,s')=(1,0)$, that is, $s=1$, and
\[
 \frac fg=(t+1)pq(1+o(1))\geq (t+1)\frac{\epsilon}t
\left(1-\frac{\epsilon}t\right)(1+o(1))
=\epsilon(1+o(1))>\beta.
\]
In both cases we have $\frac fg>\beta$, and this together 
with \eqref{eq:<beta} yields \eqref{f(g-bf)>g(bg-f)} as needed.
\end{proof}

\section{Concluding remarks}\label{sec:conclude}

Recently, Ellis, Keller and Lifshitz \cite{EKL} obtained a sharp stability result for
$t$-intersecting families. The following is Theorem 1.10 of \cite{EKL} with only minimal
changes of notation. 
\begin{theorem}[\cite{EKL}]
For any $t \in \N$  and any $\xi >0$, there exists $C=C(t, \xi)>0$ such that the following holds.
 Let $\frac{r}{t + 2r -1} + \xi < p < \frac{r+1}{t+2r+1} - \xi$, and let $\epsilon >0$. If $\cF\subset 2^{[n]}$ is a $t$-intersecting family such that 
$\mu_p(\cF)\geq (1 -\epsilon)\mu_p(\cF_r^t)$, then there exists a family $\cG$ isomorphic to $\cF_r^t$
 such that $\mu_p(\cF\setminus\cG)\leq C\epsilon^{\log_{1-p}p}$. 
\end{theorem}

 Their result is related to 
Theorem~\ref{thm:weakstability} for the case $\cA=\cB$. 
To make a comparison easier we state a version of our 
Theorem~\ref{thm:weakstability}, which can be proved almost exactly as
Theorem~\ref{thm:weakstability} is proved.
\begin{theorem}
For every integer $r\geq 0$ and all real numbers $\epsilon \in (0,1/2)$ 
and $C>2$, there exists an integer $ t_0 = t_0 (r,\epsilon, C)$ such that
for all $n\geq t\geq  t_0 $ 
the following holds.
Let 
$\frac {r+\epsilon}{t}\leq  p\leq \frac {r+1-\epsilon}{t}$. 
If $\cA,\cB\subset 2^{[n]}$ are $t$-nice families such that 
$\sqrt{\mu_p (\cA)\mu_p (\cB)}\geq (1-\gamma)\mu_p (\cF^t_r)$ 
with $\gamma\in (0,\frac{\epsilon}{2(r+1)})$, then
\[
\mu_p (\cA\triangle\cF^t_r)+\mu_p (\cB\triangle\cF^t_r)\leq C\mu_p(\cF^t_r).
\]
In particular, if $\cA=\cB$ then
$\mu_p (\cA)\geq (1-\gamma)\mu_p (\cF^t_r)$ implies
$\mu_p (\cA\triangle\cF^t_r)\leq (C/2)\mu_p(\cF^t_r)$.
\end{theorem}

Their setup is different from ours and it seems that
neither result implies the other. We should note that their results
apply to all (not necessarily shifted) $t$-intersecting families.
It would be very interesting to see whether one can use their technique to remove the
shiftedness condition from Theorem~\ref{thm:stability}.
See also \cite{EKL,Fri} for related stability results for intersecting
families.

\end{document}